[1994/06/01]

\documentclass[12pt,reqno,intlimits,a4paper,english]{amsart}

\usepackage{graphicx,pstricks,pst-3dplot}
\usepackage{calc}
\usepackage{babel}
\usepackage{amssymb}

\numberwithin{equation}{section}
\newcounter{hours}\newcounter{minutes}
\newcommand{\hourandminute}{\setcounter{hours}{\time/60}%
\setcounter{minutes}{\time-\value{hours}*60}%
\thehours.\theminutes}
\newcommand{\Versione}{\jobname\ \today\ \hourandminute}
\newlength{\Indent}
\newlength{\Parskip}
\theoremstyle{plain}
\newtheorem{thm}{Theorem}[section]     
\newtheorem{lemma}[thm]{Lemma}
\newtheorem{prop}[thm]{Proposition}

\theoremstyle{remark}

\newtheorem{Remark}[thm]{Remark}
\newenvironment{remark}{\begin{Remark}}{\qed\end{Remark}}

\theoremstyle{definition}
\newtheorem{Defin}[thm]{Definition}
\newenvironment{defin}{\begin{Defin}}{\qed\end{Defin}}

\DeclareMathOperator{\Div}{div}

\newcommand{\RN}{\Bbb{R}^{N}}
\newcommand{\R}{\Bbb{R}}
\newcommand{\abs}[1]{\lvert#1\rvert}
\newcommand{\Abs}[1]{\left | #1 \right |}
\newcommand{\norm}[1]{\lVert#1\rVert}
\newcommand{\di}{\,\text{\rmfamily\upshape d}}

\newcommand{\pder}[2]{\frac{\partial #1} {\partial #2}}
\newcommand{\norma}[2]{\norm{#1}_{#2}}

\newcommand{\Om}{\varOmega}
\newcommand{\eps}{\varepsilon}

\newcommand{\CC}{\mathcal{C}}
\newcommand{\WW}{\mathcal{W}}
\newcommand{\ZZ}{\Bbb{Z}}

\def\X{{\mathcal X}}

\newcommand{\lfint}{\lambda_{\textup{int}}}
\newcommand{\lfout}{\lambda_{\textup{out}}}
\newcommand{\lfboth}{\lambda}
\newcommand{\lfbothe}{\lambda^\eps}
\newcommand{\lfav}{\lfboth_{0}}

\newcommand{\eh}{\boldsymbol{e}_{h}}
\newcommand{\ej}{\boldsymbol{e}_{j}}

\newcommand{\Per}{E}
\newcommand{\Omint}{\Om_{\textup{int}}^{\eps}}
\newcommand{\Omout}{\Om_{\textup{out}}^{\eps}}
\newcommand{\Memb}{\varGamma^{\eps}}
\newcommand{\wMemb}{{\widehat\varGamma}^{\eps}}
\newcommand{\Perint}{\Per_{\textup{int}}}
\newcommand{\Perout}{\Per_{\textup{out}}}
\newcommand{\Permemb}{\varGamma}

\newcommand{\perzero}{H^{1}_{\#}(Y)}

\newcommand{\XX}{\X^\eps}
\newcommand{\beltrami}{\Delta^{\!\!B}}
\newcommand{\beltramigrad}{\nabla^{B}}
\newcommand{\beltramidiv}{\Div^{B}}

\newcommand{\spaziosoleps}{L^2\big(0,T;\XX_0(\Om)\big)}
\newcommand{\wto}{\rightharpoonup}

\newcommand{\cell}{Y}
\newcommand{\celloc}[1]{\cell_{\eps}#1}
\newcommand{\genfun}{w}
\newcommand{\hatfun}{\widehat{\genfun}}
\newcommand{\tildefun}{\widetilde{\genfun}}
\newcommand{\genfune}{\genfun_{\eps}}
\newcommand{\fun}{\phi}

\newcommand{\intepart}[1]{\left[#1\right]}
\newcommand{\macropart}[2]{\left[\frac{#1}{#2}\right]_{\cell}}

\newcommand{\micropart}[2]{\left\{\frac{#1}{#2}\right\}_{\cell}}

\newcommand{\intOset}{\widehat{\Om}_{\eps}}
\newcommand{\intspacetime}{\Lambda^{\eps}_T}
\newcommand{\const}{\gamma}
\newcommand{\unfop}{\mathcal{T}_{\eps}}
\newcommand{\btsunfop}{\mathcal{T}^b_{\eps}}
\newcommand{\average}{\mathcal{M}^\eps}
\newcommand{\saverage}{\mathcal{M}_{\cell}}

\newcommand{\nOmout}{\Om^{\eps,\eta}}

\newcommand{\nOmoutuno}{\Om^{\eps,\eta}_{\text{int}}}
\newcommand{\nOmoutdue}{\Om^{\eps,\eta}_{\text{out}}}
\newcommand{\nOmint}{\varGamma^{\eps,\eta}}
\newcommand{\wnOmint}{{\varGamma}^{\eps,2\eta}}

\newcommand{\nMemb}{\partial\varGamma^{\eps,\eta}}

\newcommand{\nPerout}{\Per^\eta}

\newcommand{\nPeroutuno}{\Per_{\text{int}}^{\eta}}

\newcommand{\nPeroutdue}{\Per_{\text{out}}^{\eta}}

\newcommand{\nPerint}{\Permemb^\eta}

\newcommand{\nPermemb}{\partial\nPerint}

\newcommand{\soluz}{u^\eta_\eps}

\newcommand{\soluzt}{u^\eta_{\eps t}}
\newcommand{\ndfbothe}{A^{\eta\eps}}
\newcommand{\nlfbothe}{B^{\eta\eps}}
\newcommand{\testeta}{\Phi^{\eta\eps}}

\newcommand{\yg}{y_{_{\Memb}}}
\newcommand{\tangrad}[1]{\nabla^{B}_{\Memb_{#1}}}

\makeindex
\begin{document}

\title
{Concentration and homogenization in electrical conduction in heterogeneous media involving the
      \textsc{L}aplace-\textsc{B}eltrami operator}
\author{M. Amar$^\dag$ -- D. Andreucci$^\dag$ -- R. Gianni$^\ddag$ -- C. Timofte$^\S$\\
\hfill \\
$^\dag$Dipartimento di Scienze di Base e Applicate per l'Ingegneria\\
Sapienza - Universit\`a di Roma\\
Via A. Scarpa 16, 00161 Roma, Italy
\\ \\
$^\ddag$Dipartimento di Matematica ed Informatica\\
Universit\`{a} di Firenze\\
Via Santa Marta 3, 50139 Firenze, Italy
\\ \\
$^\S$University of Bucharest\\
Faculty of Physics\\
P.O. Box MG-11, Bucharest, Romania
}

\begin{abstract}
We study a concentration and homogenization problem modelling electrical conduction in a composite material.
The novelty of the problem is due to the specific scaling of the physical quantities characterizing the dielectric
component of the composite. This leads to the appearance of a peculiar displacement current governed by a Laplace-Beltrami
pseudo-parabolic equation. This pseudo-parabolic character is present also in the homogenized equation, which is obtained
by the unfolding technique.
\medskip


  \textsc{AMS-MSC:} 35B27, 35K70, 74Q10
  \bigskip

 \textbf{Acknowledgments}: The first author is member of the \emph{Gruppo Nazionale per l'Analisi Matematica, la Probabilit\`{a} e le loro Applicazioni} (GNAMPA) of the \emph{Isti\-tuto Nazionale di Alta Matematica} (INdAM).
The second author is member of the \emph{Gruppo Nazionale per la Fisica Matematica} (GNFM) of the \emph{Istituto Nazionale di Alta Matematica} (INdAM).
The last author wishes to thank \emph{Dipartimento di Scienze di Base e Applicate per l'Ingegneria} for the warm hospitality and \emph{Universit\`{a} ``La Sapienza" of Rome} for the financial support.

\end{abstract}

\maketitle



\section{Introduction}\label{s:introduction}
The continuous development of new nano-engineered materials, as well as the search for new diagnostic devices in medicine and biology, has given big momentum to new studies in the field of composite  media, in particular with regard to their physical behavior especially from the electric conduction and heat conduction point of view. Typically, these materials are composed by a hosting medium in which nano-particles inclusions (having different physical properties) are present, commonly arranged in a periodic lattice.

The thermal behavior of these composite materials plays a fundamental role, e.g., in the design of heat dissipating fillers (e.g. for electronic devices) and in the creation  of new generation motor coolants (see \cite{Dehghani:Soni:2005,Sebnem:2010, Warangkhana:2012,Khan:2011}). Some of the authors have conducted researches in this direction (\cite{Amar:Gianni:2016B, Amar:Gianni:2016A, Amar:Gianni:2016C}) in which surface heat conduction on the interface $\Gamma$ separating the inclusions from the host material plays a fundamental role.
The case of electrical conduction was considered in \cite{Amar:Andreucci:Gianni:Timofte:2017B}.
In the models thus obtained, a non-standard surface conduction appears, governed by a Laplace-Beltrami operator.

In this article, we deal with the electric behavior of nano-composites, thus pursuing a research started by some of the authors years ago (\cite{Amar:Andreucci:Bisegna:Gianni:2003c}--%
\cite{Amar:Andreucci:Gianni:2014a}).

In that framework, the composite was used to model a biological tissue in which the hosting medium is  the extracellular material, while the inclusions are the cells separated from the surrounding medium by the lipidic cell membrane which behaves like a dielectric and, therefore, exhibits a capacitive behavior.

In those papers both the case of “thick” $N$-dimensional membranes and the one in which the membranes are regarded as electrically active, $(N-1)$-dimensional surfaces are considered. Namely, the first case is reduced to the second one via a concentration technique. Obviously, this requires specific scalings of the relevant physical quantities. The scalings used in that case were consistent with the peculiar physical situation under examination, namely we assumed that the relative dielectric constant of the ”thick” membrane scaled as $\eta$ (where $\eta$ is the relative width of the  membrane with respect to the diameter of the cell). This choice reduces the limit problem to a system of partial differential equations in which the current is continuous across the $(N-1)$-dimensional membranes and the time derivative of the jump of the potential across the membranes is proportional to the current. Such a model reproduces the standard equations satisfied, in electric conduction, by capacitors and therefore appears to be the natural model to describe the capacitive behavior of cell membranes. Finally, the $(N-1)$-dimensional model was homogenized taking into account the large number of cells present in the medium, i.e. a limit for $\eps$ going to zero of the solution was performed ($\eps$ being the typical length scale of a cell). The homogenized limit thus obtained solves an elliptic equation with memory (the memory being a direct consequence of the presence of capacitors).

In this paper, we investigate the same electrical conduction model, but with a relative dielectric constant scaling as $1/\eta$ in the thick
interface. This is consistent therefore with a strongly insulating coating as made possible by new materials (for instance, the barium titanate)
in the construction of electronic devices.



Hence, in this paper, a concentration limit is performed using the aforementioned scaling. Such a concentration limit relies, as usual, on suitable choices of test functions in the weak formulation to select the quantities appearing in the limit equations satisfied on the membranes.


Going into details, after concentration, we obtain a system of partial differential equations made of two elliptic equations satisfied in the hosting medium and in the inclusions with the potential being continuous across the interfaces, while the jump of the currents plays the role of the source term for a Laplace-Beltrami equation satisfied by the time derivative of the potential on the interfaces.

This problem is highly non-standard and, as far as we know, mathematically new.
Existence and uniqueness for the previous problem are mathematically interesting and were proved by the authors in \cite{Amar:Andreucci:Gianni:Timofte:2017B}.

Here, the concentration limit equations are homogenized letting $\eps$ go to zero (assuming a periodic arrangement of the inclusions),
via an unfolding method. It is worthwhile noticing that the unfolding limits on $\Gamma$ are quite technical and of some mathematical interest (see also \cite{Amar:Gianni:2016A}).

Owing to the linearity of the problem, the first corrector can be factorized even if the structure of the factorization is neither standard, nor simple (as was already the case in the problem studied in \cite{Amar:Andreucci:Bisegna:Gianni:2004a, Amar:Andreucci:Bisegna:Gianni:2006a, Amar:Andreucci:Bisegna:Gianni:2013}). For this reason, we are able to produce a limit equation for the macroscopic variable $u$, which contains a memory term, as in the previously quoted papers.
The main and more relevant difference with respect to the previous case, is that now the partial differential equation with memory is not always an elliptic equation,
but it substantially depends on the underlying geometry. In particular, when both the hosting medium and the inclusion are connected, we obtain
a pseudo-parabolic equation. The class of such homogenized problem will be studied by the authors in a forthcoming paper.
It is worthwhile noticing that in the present case of the dielectric constant scaling as $1/\eta$ we expect that the two steps of concentration and
homogenization commute, that is they lead to the same problem, if applied in different order. This is not expected for the problem with
dielectric constant scaling as $\eta$. An additional peculiarity of the present problem is the appearance of compatibility conditions
needed to solve the cell equations.

The results quoted above concern the case in which the permeability in the concentrated problem scales as $\eps$; however,
in order to present a complete study of the problem, we consider also all the other possible scalings $\eps^k$, with $k\in\R$.
In these cases, the resulting homogenized equation is a standard elliptic partial differential equation, not containing any memory term
(see Theorems \ref{t:casoKpiccolo_bis} and \ref{t:casoKgrande}).
Note that the concentration procedure is independent of $k$.
Moreover, for certain scalings and a particular geometry (i.e. if both the hosting medium and the inclusion are connected), the
sequence of the solutions $u_\eps$ of the approximating problems tends to $0$, independently of the presence of a non-zero source
and non-zero initial datum (see Theorem \ref{t:casoKpiccolo}).

\medskip

The paper is organized as follows. In Subsection \ref{s:LB} we briefly recall the definition and the main properties of the tangential operators (gradient, divergence, Laplace-Beltrami operator), in Subsection \ref{ss:geometric} we state our geometrical setting, in Subsection \ref{ss:unfold} we recall the main properties of the unfolding operator and, finally, in Subsection \ref{ss:position} we state the problem and our main results. Section 3 is devoted to the derivation of the concentrated problem.  In Section 4, we prove the homogenization result for the microscopic problem \eqref{eq:PDEin}--\eqref{eq:InitData} (i.e., the scaling $k=1$).
Finally, in Section 5, we consider the other scalings $\eps^k$, for $k<1$ and $k>1$.

\medskip
\goodbreak
\eject

\section{Preliminaries}\label{s:threeD_problem}
\nobreak
\subsection{Laplace-Beltrami derivatives}
\label{s:LB}
Let $\phi$ be a ${\mathcal C}^2$-function,  $\mathbf{\Phi}$ be a
${\mathcal C}^2$-vector function and $S$ a smooth surface
with normal unit vector $n$.
We recall that the tangential gradient of $\phi$ is given by
\begin{equation}\label{eq:a5bis}
\beltramigrad\phi=\nabla\phi-(n\cdot\nabla\phi)n
\end{equation}
and the tangential divergence of $\Phi$ is given by
\begin{multline}\label{eq:a3bis}
\beltramidiv\mathbf{\Phi}
=\Div\mathbf{\Phi}- (n\cdot\nabla\mathbf{\Phi}_i)n_i-(\Div n)(n\cdot\mathbf{\Phi})
\\
=\beltramidiv \left(\mathbf{\Phi}- (n\cdot\mathbf{\Phi})n\right)
=\Div \left(\mathbf{\Phi}- (n\cdot\mathbf{\Phi})n\right)\,,
\end{multline}
where, taking into account the smoothness of $S$,
the normal vector $n$ can be naturally defined in a small neighborhood of $S$ as a regular field.
Moreover, we define the Laplace-Beltrami operator as
\begin{equation}
\label{eq:beltrami}
\beltrami\phi =\Div^B(\beltramigrad\phi)\,,
\end{equation}
so that, by \eqref{eq:a5bis} and \eqref{eq:a3bis}, we get that the Laplace-Beltrami operator can be written as
\begin{multline}
\label{eq:beltrami_bis}
\beltrami\phi =
\Delta\phi-n^t\nabla^2\phi n - (n\cdot\nabla\phi)\Div n\\
= (\delta_{ij}-n_in_j)\partial^2_{ij}\phi -  n_j\partial_j\phi \partial_in_i
= (I-n\otimes n)_{ij}\partial^2_{ij}\phi-(n\cdot\nabla\phi)\Div n\,.
\end{multline}
Finally, we recall that on a regular surface $S$ with no boundary (i.e. when $\partial S=\emptyset$)
we have
\begin{equation}\label{eq:a66}
\int_S \beltramidiv \mathbf{\Phi}\di\sigma =0 \,.
\end{equation}

\subsection{Geometrical setting}\label{ss:geometric}
The typical periodic geometrical setting is displayed in Figure~\ref{fig:omega} and Figure~\ref{fig:3d}.
Here we give, for the sake of clarity, its detailed formal definition.

Let $N\geq 3$. Introduce a periodic open subset $\Per$
of $\RN$, so that $\Per+z=\Per$ for all $z\in\ZZ^{N}$.
We employ the notation $Y=(0,1)^{N}$ and
$\Perint=\Per\cap Y$, $\Perout=Y\setminus\overline{\Per}$,
$\Permemb=\partial\Per\cap \overline Y$.
We assume that $\Perout$ is connected,
while $\Perint$ may be connected or not.
As a simplifying assumption, we stipulate that $|\Permemb\cap\partial Y|_{N-1}=0$.

Let $\Om$ be an open connected bounded subset of $\RN$; we assume that $\Om$ and $E$ are of class $\CC^\infty$,
though this assumption can be weakened.
For all $\eps>0$, define $\Omint=\Om\cap\eps \Per$,
$\Omout=\Om\setminus\overline{\eps \Per}$, so that
$\Om=\Omint\cup\Omout\cup\Memb$, where $\Omint$ and $\Omout$ are
two disjoint open subsets of $\Om$, and
$\Memb=\partial\Omint\cap\Om=\partial\Omout\cap\Om$.
The region $\Omout$ [respectively, $\Omint$] corresponds to the outer
phase [respectively, the inclusions], while
$\Memb$ is the interface; in fact, these definitions are slightly modified below.
We will consider two different cases: in the first one (to which we will
refer as the {\it connected/disconnected case}, see Fig.\ref{fig:omega})
we will assume $\Permemb\cap \partial Y=\emptyset$; moreover,
we stipulate that all the cells which intersect $\partial\Om$ do not contain any inclusion.

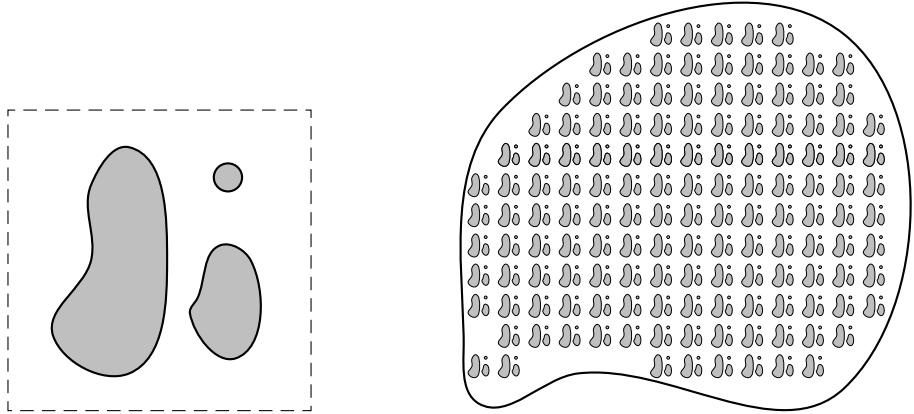
\begin{figure}[htbp]%
\begin{center}%
\begin{pspicture}(12,6)
\rput(0,1){
\psframe[linewidth=0.4pt,linestyle=dashed](0,0)(4,4)
\rput[l](-0.4,0){\psccurve[fillstyle=solid,fillcolor=lightgray](2,.5)(2.5,2)(2,3.5)(1.5,3)(1.5,2)(1,1)}
\psccurve[fillstyle=solid,fillcolor=lightgray](3,.7)(3.2,2)(2.8,2.2)(2.5,1.5)(2.4,1.3)
\pscircle[fillstyle=solid,fillcolor=lightgray](2.9,3.1){.2}
}
\newcommand{\minicell}{\scalebox{0.1}{
\rput[l](-0.4,0){\psccurve[fillstyle=solid,fillcolor=lightgray](2,.5)(2.5,2)(2,3.5)(1.5,3)(1.5,2)(1,1)}
\psccurve[fillstyle=solid,fillcolor=lightgray](3,.7)(3.2,2)(2.8,2.2)(2.5,1.5)(2.4,1.3)
\pscircle[fillstyle=solid,fillcolor=lightgray](2.9,3.1){.2}
}}
\rput(6,1){
\multirput(0,0.4)(0.4,0){2}{\minicell}%
\multirput(2.4,0.4)(0.4,0){6}{\minicell}%
\multirput(0.4,0.8)(0.4,0){12}{\minicell}%
\multirput(0,1.2)(0,0.4){5}{\multirput(0,0)(0.4,0){14}{\minicell}}%
\multirput(0.4,3.2)(0.4,0){13}{\minicell}%
\multirput(0.4,3.2)(0.4,0){13}{\minicell}%
\multirput(0.8,3.6)(0.4,0){12}{\minicell}%
\multirput(1.2,4)(0.4,0){10}{\minicell}%
\multirput(1.6,4.4)(0.4,0){9}{\minicell}%
\multirput(2.4,4.8)(0.4,0){5}{\minicell}%
\psccurve(0.2,.1)(1.5,.5)(5,.3)(5,5)(.5,4)(0,1)
}

\end{pspicture}%
    \caption{\textsl{Left}: the periodic cell $Y$. $\Perint$ is the shaded
    region and $\Perout$ is the white region.
    \textsl{Right}: the region $\Om$.}
    \label{fig:omega}
  \end{center}
\end{figure}

In the second case (to which we will refer as the {\it connected/connected case},
see Fig.\ref{fig:3d}) we will
assume that $\Perint$, $\Perout$, $\Omint$ and $\Omout$ are connected and, without loss of generality,
that they have Lipschitz continuous boundary. In this last case, we stipulate
that there exist $\gamma_1>\gamma_0>0$ such that the picture just described is actually valid in the set
$\{x\in\Om\ :\ {\rm dist}(x,\partial\Om)\geq \gamma_1\eps\}$, while
in the layer $\gamma_0\eps\leq{\rm dist}(x,\partial\Om)\leq \gamma_1\eps$,
the geometry of $\Memb$ is modified in a standard
way to keep it regular and preserve the topological properties of $\Omint$ and $\Omout$,
as well as to obtain that ${\rm dist}(\Memb,\partial\Om)\geq\gamma_0\eps$.

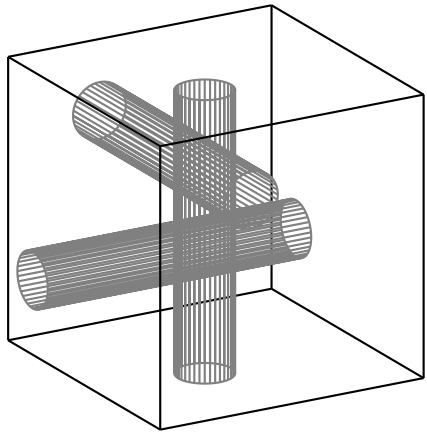
\begin{figure}[htbp]%
\begin{center}%
\begin{pspicture*}(-6,0)(6,6)
\psset{Alpha=30,Beta=20,unit=0.8cm}
\rput(0,2.5){
\pstThreeDCoor[xMin=0,xMax=5,yMin=0,yMax=5,zMin=0,zMax=5,,drawing=false,linecolor=black]
\pstThreeDLine(0,0,0)(0,0,5)
\pstThreeDLine(0,0,0)(0,5,0)
\pstThreeDLine(0,0,0)(5,0,0)
\pstIIIDCylinder[
linecolor=gray,
increment=10%
](3,3,0){0.5}{5}
\pstThreeDPut(0,3,3.5){\pstIIIDCylinder[RotY=90,
linecolor=gray,
increment=10%
](0,3,3){0.5}{5}}
\pstThreeDPut(2,0,5){\pstIIIDCylinder[RotX=-90,
linecolor=gray,
,increment=10%
](3,0,3){0.5}{5}
}
\pstThreeDLine(5,5,5)(0,5,5)
\pstThreeDLine(5,5,5)(5,5,0)
\pstThreeDLine(5,5,5)(5,0,5)
\pstThreeDLine(0,0,5)(0,5,5)
\pstThreeDLine(0,0,5)(5,0,5)
\pstThreeDLine(0,5,0)(0,5,5)
\pstThreeDLine(5,0,0)(5,0,5)
\pstThreeDLine(5,5,0)(5,0,0)
\pstThreeDLine(5,5,0)(0,5,0)
}
\end{pspicture*}%
    \caption{The periodic cell $Y$. $\Perint$ is the shaded
    region and $\Perout$ is the white region.}
    \label{fig:3d}
  \end{center}
\end{figure}

Finally, let $\nu$ denote the normal unit vector to $\Permemb$ pointing
into $\Perout$, extended by periodicity to the whole of $\R^N$, so that $\nu_\eps(x)=\nu(x/\eps)$
denotes the normal unit vector to $\Memb$ pointing into $\Omout$.

\medskip

Actually, in a more realistic framework, the isolating interface is not an
$(N-1)$-dimensional surface, but it has a very small positive thickness.
Hence, we consider also a more physical geometric setting where a small
parameter $\eta>0$ represents the small ratio between the thickness of the physical interface
and the characteristic dimension of the microstructure.
To this purpose, for $\eta>0$, let us write $\Om$ also
as $\Om=\nOmout\cup\nOmint\cup\nMemb$, where $\nOmout$ and $\nOmint$
are two disjoint open subsets of $\Om$, $\nOmint$ is the tubular neighborhood of
$\Memb$ with thickness $\eps\eta$, and $\nMemb$ is the
boundary of $\nOmint$. Moreover, we also assume
that $\nOmout={\nOmoutuno}\cup{\nOmoutdue}$, where
$\nOmoutdue\subset\Omout$, $\nOmoutuno\subset\Omint$
and $\nMemb =(\partial{\nOmoutuno}\cup\partial{\nOmoutdue})\cap \Om$.
We notice that, for $\eta\to 0$ and
$\eps>0$ fixed, $|\nOmint|\sim \eps\eta|\Memb|_{N-1}$.
Set also $Y= \nPeroutuno\cup \nPeroutdue\cup\nPerint\cup\nPermemb$
where $\nPeroutuno,\nPeroutdue$ and $\nPerint$ are disjoint open subsets of
$Y$, $\nPerint$ is the tubular neighborhood of
$\Permemb$ with thickness $\eta$, and $\nPermemb=(\partial \nPeroutuno\cup\partial\nPeroutdue)\cap Y$
(see Figure~\ref{fig:periods}).
For the sake of brevity, we will denote by $\nPerout$ the union $\nPeroutuno\cup \nPeroutdue$.
Finally, for $\eta\to 0$, $|\nPerint|\sim \eta |\Permemb|_{N-1}$.

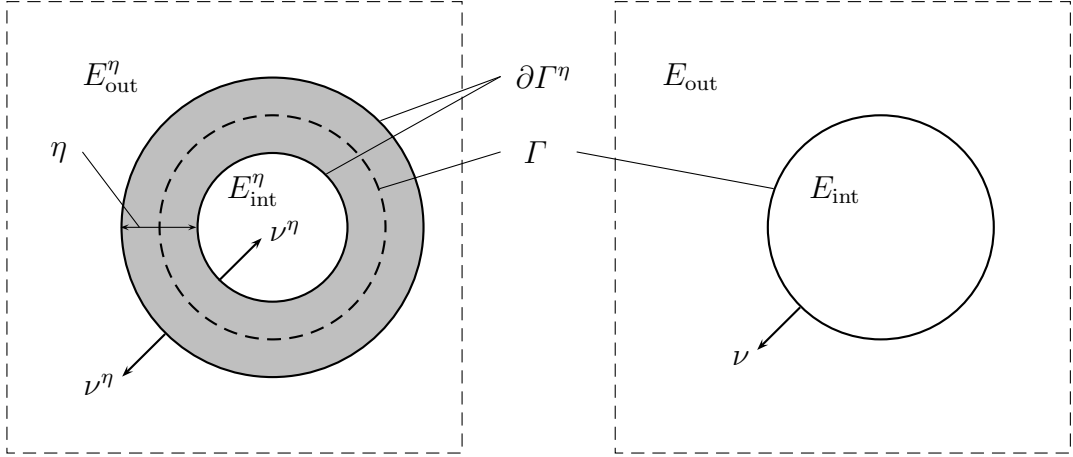
\begin{figure}[htbp]
  \begin{center}
    \begin{pspicture}(12,7)
\rput(-1,0.5){
\psframe[linewidth=0.4pt,linestyle=dashed](0,0)(6,6)
\pscircle[fillstyle=solid,fillcolor=lightgray](3.5,3){2}
\pscircle[linestyle=dashed](3.5,3){1.5}
\pscircle[fillstyle=solid,fillcolor=white](3.5,3){1}
\psline[linewidth=0.4pt]{<->}(1.5,3)(2.5,3)
\psline[linewidth=0.4pt](1.75,3)(1,4)
\psline[linewidth=0.4pt](4.9192,4.4142)(6.5,5)
\psline[linewidth=0.4pt](4.2071,3.7071)(6.5,5)
\psline[linewidth=0.4pt](4.9095,3.513)(6.5,4)
\psline{->}(2.088,1.5858)(1.5223,1.0201)
\psline{->}(2.7929,2.2918)(3.3586,2.8575)
\rput[l](1,5){\mbox{$E^{\eta}_{\textup{out}}$}}
\rput[l](2.9,3.5){\mbox{$E^{\eta}_{\textup{int}}$}}
\rput[l](6.7,5){\mbox{$\partial \varGamma^{\eta}$}}
\rput[l](6.8,4){\mbox{$\varGamma$}}
\rput[r](0.8,4){\mbox{$\eta$}}
\rput[r](1.4223,.9201){\mbox{$\nu^{\eta}$}}
\rput[l](3.4586,2.9575){\mbox{$\nu^{\eta}$}}
\rput(8,0){
\psframe[linewidth=0.4pt,linestyle=dashed](0,0)(6,6)
\pscircle(3.5,3){1.5}
\psline[linewidth=0.4pt](2.0905,3.513)(-.5,4)
\psline{->}(2.4393,1.9393)(1.8736,1.3736)
\rput[r](1.7736,1.2736){\mbox{$\nu$}}
\rput(1,5){\mbox{$E_{\textup{out}}$}}
\rput(2.9,3.5){\mbox{$E_{\textup{int}}$}}
}
}
\end{pspicture}
    \caption{The periodic cell $Y$. \textsl{Left}: before concentration;
    $\varGamma^{\eta}$ is the shaded
    region, and $\nPerout=\nPeroutuno\cup \nPeroutdue$ is the white region.
    \textsl{Right}: after concentration; $\varGamma^{\eta}$
    shrinks to $\varGamma$ as $\eta\to0$.}
    \label{fig:periods}
  \end{center}
\end{figure}

\medskip

We stress the fact that the appearance of the two small parameters $\eta$ and $\eps$
 calls for two different limit procedures: we will first perform
a concentration of the thin membranes, in order to simplify the geometrical setting
of the microstructure, and then we will perform a homogenization limit, in order
to obtain a macroscopic model.

\subsection{Definition and main properties of the unfolding operator}\label{ss:unfold}

In this subsection, we define and collect some properties of a space-time version
of the space unfolding operator introduced and developed in \cite{Cioranescu:Damlamian:Griso:2002,
 Cioranescu:Damlamian:Griso:2008, Cioranescu:Donato:Zaki:2006B, Cioranescu:Donato:Zaki:2006A}
 (see also \cite{Amar:Gianni:2016A}).

A space-time version of the unfolding operator in a more general framework,
in which also a time-microscale is actually present,  has been introduced
in \cite{Amar:Andreucci:Bellaveglia:2015B} and \cite{Amar:Andreucci:Bellaveglia:2015A},
to which we also refer for a survey on this topic.

However, in the unfolding technique used here, the
time variable does not play any special role and can be treated essentially as
a parameter, hence most of the properties of this operator can be proven
essentially as in the above quoted papers and are therefore only recalled here.
An analogous remark is valid
for the other operators which will be introduced in the following.

For the sake of simplicity, for any spatial domain $G$, we
will denote by $G_{T}=G\times(0,T)$ the corresponding space--time
cylindrical domain over the time interval $(0,T)$.

Let us set
\begin{equation*}
  \Xi_{\eps}
  =
  \left\{\xi\in \mathbb{Z}^N\,, \quad \eps(\xi+Y)\subset\Om\right\}
  \,,
  \quad
  \intOset
  =
  \text{interior}\left\{\bigcup_{\xi\in \Xi_{\eps}} \eps(\xi+\overline{\cell})   \right\}
  \,,
  \end{equation*}
  \begin{equation*}
  \Om_T =\Om\times (0,T)
  \,,
  \quad
    \intspacetime
  =
  \intOset\times(0,T)
  \,.
\end{equation*}

Denoting by $[r]$ the integer part of $r\in\R$, we define for $x\in\RN$
\begin{equation*}
  \macropart{x}{\eps}
  =
  \Big(
  \intepart{\frac{x_{1}}{\eps}}
  ,
  \dots
  ,
  \intepart{\frac{x_N}{\eps}}
  \Big)
  \,,\quad\text{so that}\qquad
  x
  =
  \eps\left(\macropart{x}{\eps}+\micropart{x}{\eps}\right)
  \,.
\end{equation*}
Then, we introduce the space cell containing $x$ as $  \celloc{(x)} =  \eps\Big(\displaystyle\macropart{x}{\eps}
  +  \cell  \Big)$.
Finally, set $y_M=y-\int_Y y\di y$.

\begin{defin}\label{d:oldunfop}
For a Lebesgue-measurable $\genfun$ on $\Om_T$, the (time-depending) periodic
unfolding operator $\unfop$ is defined as
\begin{equation*}
  \unfop(\genfun)(x,t,y)
  =
  \left\{
    \begin{alignedat}2
      &\genfun \left(\eps\macropart{x}{\eps}+\eps y, t \right),
      &\quad&
      (x,t,y)\in \intspacetime\times\cell;
      \\
      &0 \,,
      &\quad&
      \text{otherwise.}
    \end{alignedat}
  \right.
\end{equation*}
For a Lebesgue-measurable $\genfun$ on $\Memb_T$, the (time-depending) boundary unfolding operator $\btsunfop$
is defined as
\begin{equation*}
  \btsunfop(\genfun)(x,t,y)
  =
  \left\{
    \begin{alignedat}2
      &\genfun \left(\eps\macropart{x}{\eps}+\eps y, t \right),
      &\quad&
      (x,t,y)\in \intspacetime\times\Permemb;
      \\
      &0 \,,
      &\quad&
      \text{otherwise.}
    \end{alignedat}
  \right.
\end{equation*}
\end{defin}

Clearly, for $\genfun_{1}$, $\genfun_{2}$ as in Definition~\ref{d:oldunfop}
\begin{equation}
  \label{eq:unfop_product}
  \unfop(\genfun_{1}\genfun_{2})
  =
  \unfop(\genfun_{1})
  \unfop(\genfun_{2})
  \,,
\end{equation}
and the same property holds for the boundary unfolding operator.
Note that $\btsunfop(\genfun)$ is the trace of the unfolding operator
on $\intspacetime\times\Permemb$, when both operators are defined.
\medskip

\begin{defin}\label{d:average}
For a Lebesgue-measurable $\genfun$ on $\Om_T$, the (time-depending) local average operator
is defined as
 \begin{equation}
   \label{eq:local_s}
   \average(\genfun)(x,t)
   =
   \left\{
   \begin{alignedat}2
     &\frac{1}{\eps^N}
     \int_{\celloc{(x)}}
     \genfun(\zeta,{t})
     \di \zeta
     \,,
     &\quad&
     \text{if}
     \,
     (x,t)\in
     \intspacetime
     \,,
     \\
     &0
     \,,
     &\quad&
     \text{otherwise.}
   \end{alignedat}
   \right.
 \end{equation}
\end{defin}
By a change of variable, it is not difficult to see that
$$
 \average(\genfun)(x,t)= {\mathcal M}_Y\big(\unfop(\genfun)\big)\,,
$$
where we denote by $\saverage(\cdot)$ the integral average on $Y$.

We collect here some properties of the operators defined above.

\begin{prop}
  \label{p:norms}
  The operator $\unfop:L^{2}(\Om_T)\to L^{2}(\Om_T\times\cell)$ is
  linear and continuous.
  In addition, we have
   \begin{equation}
   \label{eq:norm_bound}
   \norma{\unfop(\genfun)}{L^2(\Om_T\times\cell)}
   \le
   \norma{\genfun}{L^2(\Om_T)}
   \,,
  \end{equation}
  and
  \begin{equation}
    \label{eq:norms_n}
    \Abs{
      \int_{\Om_T}
      \genfun
      \di x
      \di t
      -
      \iint_{\Om_T\times\cell}
      \unfop(\genfun)
  \di y
      \di x
      \di t
    }
    \le
    \int_{\Om_T\setminus\intspacetime}
    \abs{\genfun}
    \di x
    \di t
    \,.
  \end{equation}
\end{prop}

\begin{prop}  \label{p:convergences}
  Let $\{\genfune\}$ be a sequence of functions in $L^2(\Om_T)$.
  \\
  If $\genfune\to\genfun$ strongly in $L^2(\Om_T)$ as $\eps\to 0$,
  then
  \begin{equation}
    \label{eq:strong_conv}
    \unfop(\genfune)\to
    \genfun
    \,,
    \quad
    \text{strongly in}
    \,
    L^2(\Om_T \times\cell)
    \,.
    \end{equation}
    If $\genfune\to\genfun$ strongly in $L^2(0,T;H^1(\Om))$ as $\eps\to 0$,
  then
  \begin{equation}
    \label{eq:a23}
     \frac{\unfop(\genfune)-\average(\genfune)}{\eps}\to
    y_M\cdot\nabla\genfun
    \,,
    \quad
    \text{strongly in}
    \,
    L^2(\Om_T)
    \,.
  \end{equation}
  If $\genfune$ is a bounded sequence of functions in $L^2(\Om_T)$,
  then, up to a subsequence
  \begin{equation}
    \label{eq:weak_conv}
    \unfop(\genfune)
    \wto
    \hatfun
    \,,
    \quad
    \text{weakly in}
    \,
    L^2\left(\Om_T\times\cell\right)
    \,,
  \end{equation}
  and
  \begin{equation}
  \label{eq:weak_conv_ii}
     \genfune
     \wto
     \saverage(\hatfun)
     \,,
   \quad
   \text{weakly in}
   \,
   L^2(\Om_T)
   \,.
  \end{equation}
\end{prop}

\begin{remark}\label{r:r6}
In particular, if $\genfun\in L^2(\Om_T)$, we get that $\unfop(\genfun)\to\genfun$,
for $\eps\to 0$, strongly in $L^2(\Om_T\times Y)$.
\end{remark}

\begin{remark}
  \label{r:fcapac_str_conv}
  We note that the only cases in which \eqref{eq:strong_conv} it is known to hold
  without assuming the strong convergence of
the sequence $\{\genfune\}$ is when $\genfune(x,t)=\fun(x,t,\eps^{-1}x)$ where $\fun$ corresponds to one of the
  following cases (or sum of them):
  $\fun(x,t,y)=f_1(x,t)f_2(y)$, with $f_1f_2\in L^1(\Om_T\times\cell)$,
  $\fun\in L^1(\cell;\CC(\Om_T))$, $\fun\in
  L^1(\Om_T;\CC(\cell))$. In all such cases we have
  $\unfop(\genfune)\to\phi$ strongly in $L^2(\Om_T\times\cell)$ (see, for instance,  \cite{Allaire:1992,Cioranescu:Damlamian:Griso:2002,Cioranescu:Damlamian:Griso:2008}
  and \cite[Remark 2.9]{Amar:Andreucci:Bellaveglia:2015A}).
\end{remark}

\begin{prop}
  \label{p:ts_norms}
  The operator $\btsunfop:L^{2}(\Memb_T)\to L^{2}(\Om_T\times\Permemb)$ is
  linear and continuous.
  In addition, we have
   \begin{equation}
   \label{eq:ts_norm_bound}
   \norma{\btsunfop(\genfun)}{L^2(\Om_T\times\Permemb)}
   \le   {\sqrt\eps}\norma{\genfun}{L^2(\Memb_T)}
   \,,
  \end{equation}
  and
  \begin{equation}
    \label{eq:ts_norms_n}
      \int_{\Memb_T}
      \genfun
      \di \sigma
      \di t
      =
      \frac{1}{\eps}
      \int_{\Om_T\times\Permemb}
      \btsunfop(\genfun)
      \di \sigma
      \di x
      \di t \,.
  \end{equation}
\end{prop}

Note that \eqref{eq:ts_norms_n} holds since we can choose $\const_0$
in Subsection \ref{ss:geometric} in such a way that
$\Memb_T\setminus\intspacetime=\emptyset$.

\begin{prop}\label{p:p2}\cite[Proposition 5]{Amar:Gianni:2016A}
Assume that $\genfune\wto\genfun$ weakly in $L^2(0,T;H^1_0(\Om))$. Then,
$$
\btsunfop(\genfune)\wto \genfun\,,\qquad\text{ weakly in $L^2(\Om_T\times\Permemb)$.}
$$
\end{prop}

Finally, we state some results which will be mainly used when we deal with testing functions.
\bigskip

\begin{prop}\label{c:c1}\cite[Corollary 1]{Amar:Gianni:2016A}
 Let $\genfun$ be a function belonging to
  $L^2\big(0,T;H^1(\Om)\big)$. Then, as $\eps\to 0$,
    \begin{equation}
     \label{eq:a82}
   \btsunfop(\genfun)\to
   \genfun
   \,,
   \quad
   \text{strongly in}
   \,
    L^2\left(\Om_T\times\Permemb\right)\,.
  \end{equation}
\end{prop}

\begin{prop}
  \label{p:per_odc_fun}
  Let $\fun:\cell\to \R$ be a function extended by $\cell$-periodicity to the
  whole of $\RN$ and define the sequence
  \begin{equation}\label{eq:a78}
    \fun^\eps(x)
    =
    \fun\left(\frac{x}{\eps}\right)
    \,,
    \qquad
    x\in \RN
    \,.
  \end{equation}
 If $\phi$ is measurable on $\cell$, then
  \begin{equation}
    \label{eq:per_osc_fun}
    \unfop(\fun^\eps)(x,y)
    =
    \left\{
      \begin{alignedat}2
        &\fun(y)
        \,,
        &\qquad&
        (x,y)\in \intOset\times\cell
        \,,
        \\
        &0 \,,
        &\qquad &
        \text{otherwise.}
      \end{alignedat}
    \right.
  \end{equation}
 Analogously,  if $\fun$ is measurable on $\Permemb$, then
  \begin{equation}
    \label{eq:a80}
    \btsunfop(\fun^\eps)(x,y)
    =
    \left\{
      \begin{alignedat}2
        &\fun(y)
        \,,
        &\qquad&
        (x,y)\in \intOset\times\Permemb
        \,,
        \\
        &0 \,,
        &\qquad &
        \text{otherwise.}
      \end{alignedat}
    \right.
  \end{equation}
   Moreover, if $\fun\in L^2(\cell)$, as $\eps\to 0$,
  \begin{equation}
    \label{eq:per_osc_fun_ii}
    \unfop(\fun^\eps)
    \to
    \fun
    \,,
    \qquad
    \text{strongly in}
    \,
    L^2(\Om\times\cell)
    \,;
  \end{equation}
  if $\fun\in L^2(\Permemb)$, as $\eps\to 0$,
  \begin{equation}
    \label{eq:a83}
    \btsunfop(\fun^\eps)
    \to
    \fun
    \,,
    \qquad
    \text{strongly in}
    \,
    L^2(\Om\times\Permemb)
    \,;
    \end{equation}
    if $\fun\in H^{1}(\cell)$, as $\eps\to 0$,
      \begin{gather}
    \label{eq:unf_y_grad_per}
    \nabla_{y}(\unfop(\fun^\eps))
    \to
    \nabla_{y}\fun
    \,,
    \quad
    \text{strongly in}
    \,
    L^2(\Om\times\cell)\,.
  \end{gather}
\end{prop}

Now, let us state some properties concerning the behavior of the unfolding operator with respect to gradients.
To this aim, we denote by $H^1_\#(\cell)$ the space of those $\cell$-periodic
functions belonging to $H^1_{loc}(\RN)$.
\bigskip

\begin{thm}
  \label{t:smalleps_grad_weak_conv}
  Let $\{\genfune\}$ be a sequence converging weakly to $\genfun$ in
  $L^2\big(0,T;H^1_0(\Om)\big)$. Then, up to a subsequence, there
  exists $\tildefun=\tildefun(x,y,t)\in
  L^2\big(\Om_T;H^{1}_{\#}(\cell))$,
  $\saverage(\tildefun)=0$, such that, as $\eps\to 0$,
  \begin{alignat}{2}
     \label{eq:a17}
&    \unfop(\genfune)
    \wto
    \genfun
    \,,
 &   \quad
  &  \text{weakly in}
    \,
    L^2(\Om_T\times\cell)
    \,,
    \\
    \label{eq:smalleps_grad_weak_conv_i}
&    \unfop(\nabla\genfune)
    \wto
    \nabla_x\genfun
    +\nabla_y \tildefun
    \,,
 &   \quad
  &  \text{weakly in}
    \,
    L^2(\Om_T\times\cell)
    \,.
  \end{alignat}
  \end{thm}

\begin{proof}
Whereas the convergence in \eqref{eq:smalleps_grad_weak_conv_i} 
is a well-known property (see for instance \cite{Amar:Gianni:2016A,Donato:Yang:2012}),
in order to prove \eqref{eq:a17}, we proceed as follows. Since $\genfune\wto\genfun$ weakly
in $L^2\big(0,T;H^1_0(\Om)\big)$, we have
that, for every test function $\theta\in L^2(0,T)$,
\begin{equation}\label{eq:a18}
\int_0^T\genfune\theta(t)\di t\to \int_0^T\genfun\theta(t)\di t\,,
\qquad \text{strongly in $L^2(\Om)$.}
\end{equation}
Therefore, by \cite[Theorem 3.5]{Cioranescu:Damlamian:Griso:2008},
\begin{equation}\label{eq:a19}
\unfop\left(\int_0^T \genfune\theta(t)\di t\right)\to \int_0^T\genfun\theta(t) \di t\,,
\qquad \text{strongly in $L^2(\Om\times\cell)$.}
\end{equation}
Moreover, by \eqref{eq:weak_conv}, there exists $\hatfun\in L^2(\Om_T\times Y)$, such that
$\unfop(\genfun)\wto\hatfun$ weakly in $L^2(\Om_T\times Y)$ and $\saverage(\hatfun)=\genfun$.
Hence, for every $\phi\in L^2(\Om)$ and every $\psi\in L^2_{\#}(\cell)$,
\begin{multline*}
\int_{\Om_T\times \cell}\hatfun\phi(x)\psi(y)\theta(t)\di x\di y\di t\leftarrow
\int_{\Om_T\times \cell}\unfop(\genfune)\phi(x)\psi(y)\theta(t)\di x\di y\di t
\\
=\int_{\Om\times \cell}\unfop\left(\int_0^T\genfune\theta(t)\di t\right)\phi(x)\psi(y)\di x\di y
\to \int_{\Om\times \cell}\left(\int_0^T\genfun\theta(t)\di t\right)\phi(x)\psi(y)\di x\di y\,,
\end{multline*}
which implies that $\hatfun=\genfun$ a.e. in $\Om_T \times \cell$. Therefore, $\unfop(\genfune)\wto \genfun$
weakly in $L^2(\Om_T\times\cell)$ and \eqref{eq:a17} is proven.
\end{proof}

\begin{prop}\label{p:p1}\cite[Proposition 8]{Amar:Gianni:2016A}
Let $\{\genfune\}$ be a sequence in $L^2\big(0,T;\XX_0(\Om)\big)$ (see \eqref{eq:space2} below, for the definition of
the space $\XX_0(\Om)$)
converging weakly to $\genfun$
  in $L^2\big(0,T;H^1_0(\Om)\big)$, as $\eps\to 0$, and such that
 \begin{equation}\label{eq:a74}
\eps \int_0^T\!\!\int_{\Memb} |\beltramigrad\genfune|^2\di\sigma\di t\leq \const\,,
\end{equation}
where $\const>0$ is a constant independent of $\eps$.
  Then, for the same function
  $\tildefun$  as in \eqref{eq:smalleps_grad_weak_conv_i},
  we have that $\beltramigrad_y\tildefun\in L^2(\Om_T\times\Permemb)$ does exist and
    \begin{equation}\label{eq:a76}
\btsunfop(\beltramigrad\genfune)\wto    \beltramigrad_x\genfun
    +\beltramigrad_y\tildefun
    \,,
    \quad
    \text{weakly in}
    \,
    L^2(\Om_T\times\Permemb)\,.
   \end{equation}
\end{prop}

\bigskip

\subsection{Setting of the problem}\label{ss:position}
In this subsection, we will present both the physical problem involving thick membranes and
the concentrated version involving only $(N-1)$-dimensional interfaces. It will be
the purpose of the next section to show that the concentration limit ($\eta\to 0$) of the physical
model actually gives rise to the mathematical microscopic scheme.
\medskip

Let $\lfint,\lfout, \alpha$ be strictly positive constants.
We give here a complete formulation of the problems stated in the Introduction
(the operators $\Div$ and $\nabla$, as well as $\Div^B$ and $\beltramigrad$, act
only with respect to the space variable $x$).
\medskip

We first state the physical problem with thick membranes. To this purpose,
we set $\ndfbothe (x) = \lfint$ in $\nOmoutuno$, $\ndfbothe (x) = \lfout$ in
$\nOmoutdue$,  $\ndfbothe (x) = 0$ in $\nOmint$,
$\nlfbothe (x) = 0$ in $\nOmoutuno\cup\nOmoutdue$,
$\nlfbothe (x) = {\alpha}/{\eta} $ in $\nOmint$.
Moreover, let $\overline u_0\in H^{1}_0(\Om)$ be a given function.
For every $\eps,\eta>0$, we consider the problem for
$\soluz\in L^2\big(0,T;H^1_0(\Om)\big)$ given by
\begin{alignat}2
  \label{eq1}
  -\Div(\ndfbothe \nabla\soluz+\nlfbothe \nabla \soluzt)&=0\,,&\qquad &\text{in $\Om_T$;}
  \\
  \label{eq5}
 \nabla\soluz(x,0)&=\nabla\overline u_0(x)\,,&\qquad&\text{in $\nOmint$,}
\end{alignat}
which has the following weak formulation
\begin{equation}\label{eq:a14}
\int_{0}^{T}\!\!\int_{\Om} \ndfbothe \nabla\soluz\cdot\nabla{\Phi}\di x\di t-
  \int_{0}^{T}\!\!\int_{\Om} \nlfbothe \nabla \soluz\cdot\nabla\Phi_t \di x\di t=
  \int_{\Om} \nlfbothe \nabla\overline u_0\cdot\nabla{\Phi}(0)\di x\,,
\end{equation}
for every test function $\Phi\in \CC^\infty(\overline\Om_T)$ such that $\Phi$ has compact support in $\Om$ for every
$t\in[0,T)$ and $\Phi(\cdot,T)=0$ in $\Om$.
By \cite[Theorem 2.3]{Amar:Andreucci:Gianni:Timofte:2017B}, for any given $\eta>0$,
problem \eqref{eq1}--\eqref{eq5} (or \eqref{eq:a14}) has a unique solution
$\soluz\in L^2\big(0,T;H^1_0(\Om)\big)$.
\medskip

Now, let us state the concentrated problem. To this purpose,
we define
$\lfbothe:\Om\to\R$ as
\begin{equation*}
  \lfbothe=\lfint\,\quad \text{in $\Omint$,}\qquad \lfbothe=\lfout\,\quad
  \text{in $\Omout$.}
\end{equation*}
For every $\eps>0$, we consider the problem
for $u_{\eps}(x,t)$ given by
\begin{alignat}2
  \label{eq:PDEin}
  -\Div(\lfbothe \nabla u_{\eps})&=0\,,&\qquad &\text{in $(\Omint\cup\Omout)\times(0,T)$;}
  \\
  \label{eq:FluxCont}
  [u_{\eps}] &=0 \,,&\qquad &\text{on
  $\Memb_T$;}
  \\
  \label{eq:Circuit}
  -\eps\alpha\beltrami u_{\eps t}
  &=[\lfbothe \nabla u_\eps  \cdot \nu_\eps]\,,&\qquad
  &\text{on $\Memb_T$;}
  \\
  \label{eq:BoundData}
  u_{\eps}(x,t)&=0\,,&\qquad&\text{on $\partial\Om\times(0,T)$;}
  \\
  \label{eq:InitData}
 \beltramigrad u_{\eps}(x,0)&=\beltramigrad\overline u_0(x)\,,&\qquad&\text{on $\Memb$,}
\end{alignat}
where we denote
\begin{equation}
  \label{eq:jump}
  [u_{\eps}] = u_{\eps}^{\text{out}}  -  u_{\eps}^{\text{int}} \,,
\end{equation}
and the same notation is employed also for other quantities.

Since problem \eqref{eq:PDEin}--\eqref{eq:InitData} is not standard, in order to define a proper notion of weak
solution, we need to introduce some suitable function spaces.
To this purpose and
for later use, we denote by $H^1(\Memb)$ the space of Lebesgue measurable functions
$u:\Memb\to\R$ such that $u\in L^2(\Memb)$, $\beltramigrad u \in L^2(\Memb)$.
Let us also set
\begin{equation}
 \label{eq:space2}
 \XX_0(\Om) :=\{u\in H^1_0(\Om)\ :\ u_{\mid\Memb}\in H^1(\Memb)\}\,.
\end{equation}

\begin{defin}\label{d:weak_sol}
Assume that $\overline u_0\in \XX_0(\Om)$.
We say that $u_\eps\in \spaziosoleps$ is a weak solution of problem \eqref{eq:PDEin}--\eqref{eq:InitData} if
\begin{multline}\label{eq:weak_sol}
  \int_{0}^{T}\!\!\int_{\Om} \lfbothe \nabla u_{\eps}\cdot\nabla\Phi \di x\di t
  -{\eps\alpha}\int_0^T\!\!\int_{\Memb} \beltramigrad u_{\eps}\cdot \beltramigrad \Phi_t \di\sigma\di t
  ={\eps\alpha}\int_{\Memb} \beltramigrad\overline u_0\cdot\beltramigrad\Phi(x,0)\di\sigma \,,
\end{multline}
for every test function $\Phi\in \CC^\infty(\overline\Om_T)$ such that $\Phi$ has compact support in $\Om$ for every
$t\in[0,T)$ and $\Phi(\cdot,T)=0$ in $\Om$.
\end{defin}

For every $\eps>0$, by \cite[Theorem 2.8]{Amar:Andreucci:Gianni:Timofte:2017B} the
problem \eqref{eq:PDEin}--\eqref{eq:InitData} admits a unique solution $u_\eps\in L^2\big(0,T;\XX_0(\Om)\big)$.
\medskip

If $u_\eps$ is smooth, by \eqref{eq:beltrami_bis} it follows that equation \eqref{eq:Circuit} can be written in the form
\begin{equation}
\label{eq:Circuitnew}
-\eps\alpha\left( \Delta u_{\eps t}- \nu_\eps^t\nabla^2 u_{\eps t}\nu_\eps -
(\nu_\eps\cdot\nabla u_{\eps t})\Div \nu_\eps\right)
 =[\lfboth \nabla u_\eps  \cdot \nu_\eps]\,,\qquad  \text{on $\Memb$,}
\end{equation}
where $\nabla^2u_{\eps t}$ stands for the Hessian matrix of $u_{\eps t}$.
\medskip

Finally, it will be useful in the sequel to define also
$\lfboth:Y\to\R$ as
\begin{equation*}
  \lfboth=\lfint\,\quad \text{in $\Perint$,}\qquad \lfboth=\lfout\,\quad
  \text{in $\Perout$.}
\end{equation*}

\begin{remark}\label{r:r5}
We have assumed, for the sake of simplicity,
that equations \eqref{eq1} and \eqref{eq:PDEin} are homogeneous,
but essentially in the same way we can treat also the case where a source
$f\in L^2(\Om_T)$ appears in \eqref{eq1}, so that it occurs,
after the concentration, also in \eqref{eq:PDEin}. However, in this case, if $f$
does not satisfy some stronger regularity condition, we cannot
expect that the solutions $u^\eta_\eps$ and $u_\eps$ admit a time-derivative
belonging to $L^2$, so that the second term in \eqref{eq1}, as well as the left-hand side
in \eqref{eq:Circuit}, should be considered in a weak sense as above (see
\cite[Remark 2.9]{Amar:Andreucci:Gianni:Timofte:2017B}). In the homogeneous case, actually,
one could show that a stronger formulation is possible.
\end{remark}

The main result of the paper is the following.

\begin{thm}\label{strong-form}
Assume that $\overline u_0\in H^1_0(\Om)\cap H^2(\Om)$.
The unique solution   $u_\eps$ of the variational problem \eqref{eq:weak_sol} converges,
in the sense of Lemma \ref{conv}, to $(u,w)\in {\mathcal{V}}$ (${\mathcal{V}}$
is the function space defined in Theorem \ref{th-conv-unfold}), where
$u$ is the unique solution of the  homogenized problem
\begin{equation} \label{homog-pb}
 \begin{aligned}
-&\Div \Big(C^0 \nabla u_t +(\lfav I+ A^0) \nabla u
\!+\!\!\! \int_0^t \!\!\!B^0 (t-\tau) \nabla u(x, \tau) \di \tau\Big)
 = F, &\!\!\!\!\!\!\!\!\!\!\!\!\!\!\!\text{in $\Om_T$;}
 \\
& u =0\,, &\!\!\!\!\!\!\!\!\!\!\!\!\!\!\!\!\!\!\!\!\!\!\!\!\!\text{on $\partial \Om \times (0,T)$;}
\\
 -&\Div\left(C^0\nabla u(x,0)\right) =-\Div\left(C^0\nabla\overline{u}_0(x)\right)
\,, &\!\!\!\!\!\!\!\!\!\!\!\!\!\!\!\!\!\!\!\!\!\!\!\!\!\text{in $\Om$.}
 \end{aligned}
\end{equation}
The corrector $w$ in \eqref{unfold} can be factorized as
\begin{equation} \label{corrector}
w(x,y,t)=\chi_0^i (y) \, \partial_i u(x,t)+\displaystyle \int_0^t \chi_1^i (y, t-\tau) \, \partial_i u(x, \tau) \di \tau +W(x,y,t).
\end{equation}
The matrices $A^0$, $B^0$ and $C^0$ are given by \eqref{matrix-A}- \eqref{matrix-C}, the functions $\chi_0^i$  and $\chi_1^i$  are defined in \eqref{cell-chi0_1}--\eqref{cell-chi0_3} and, respectively, in \eqref{cell-chi1_1}--\eqref{cell-chi1_init},
while the function $W$ is given by \eqref{cell-W_1}--\eqref{cell-W_4}. The source term $F$ is defined by \eqref{F} and
$$
\lambda_0= \lfint \vert \Perint \vert +\lfout \vert \Perout \vert.
$$
\end{thm}

Notice that the problem \eqref{homog-pb} must be intended in the following weak sense
\begin{multline}\label{eq:pb_weak}
-\int_0^T\!\!\int_\Om C^0\nabla u\cdot \nabla \Phi_t\di x\di t
+\int_0^T\!\!\int_\Om(\lfav I+A^0)\nabla u\cdot \nabla\Phi\di x\di t
\\
+\int_0^T\!\!\int_\Om\left(\int_0^t B^0(t-\tau)\nabla u(x,\tau)\di\tau\right)\cdot\nabla \Phi\di x\di t
\\
=\int_0^T\!\!\int_\Om F\Phi\di x\di t+\int_\Om C^0\nabla\overline u_0\cdot\nabla\Phi(0)\di x\,,
\end{multline}
for every test function $\Phi\in \CC^\infty(\overline\Om_T)$ such that $\Phi$ has compact support in $\Om$ for every
$t\in[0,T)$ and $\Phi(\cdot,T)=0$ in $\Om$.
In particular, the initial condition in \eqref{homog-pb} comes from the proof of the Theorem \ref{strong-form} (see \eqref{homog-1_10});
moreover, we remark that the order of derivation in \eqref{homog-pb} (as well as in \eqref{eq:Circuit} above or \eqref{eq:Circuitk} below)
is not trivial
and that $-\Div (C^0 \nabla u_t)$ should be rewritten in the form
$-\Div \left((C^0 \nabla u)_t\right)$ or $-\partial_t\left(\Div (C^0 \nabla u)\right)$
(see \cite[Remark 2.9]{Amar:Andreucci:Gianni:Timofte:2017B}). However, the weak formulation \eqref{eq:pb_weak} is rigorous.
\medskip

When different scalings with respect to $\eps$
are present, we consider, for $k\in\R$, the problem
\begin{alignat}2
  \label{eq:PDEink}
  -\Div(\lfbothe \nabla u_{\eps})&=f\,,&\qquad &\text{in $(\Omint\cup\Omout)\times(0,T)$;}
  \\
  \label{eq:FluxContk}
  [u_{\eps}] &=0 \,,&\qquad &\text{on
  $\Memb_T$;}
  \\
  \label{eq:Circuitk}
  -\eps^k\alpha\beltrami u_{\eps t}
  &=[\lfbothe \nabla u_\eps  \cdot \nu_\eps]\,,&\qquad
  &\text{on $\Memb_T$;}
  \\
  \label{eq:BoundDatak}
  u_{\eps}(x,t)&=0\,,&\qquad&\text{on $\partial\Om\times(0,T)$;}
  \\
  \label{eq:InitDatak}
 \beltramigrad u_{\eps}(x,0)&=\beltramigrad\overline u_{0\eps}(x)\,,&\qquad&\text{on $\Memb$,}
\end{alignat}
where we are interested in keeping a nonzero source $f\in L^2(\Om_T)$,
in order to show that the following results are non trivial.

\begin{thm}\label{t:casoKpiccolo}
Assume that $f\in L^2(\Om)$ and $ \overline u_{0\eps} = \eps^{(1-k)/2}\overline u_0$, with $\overline u_0\in
H^1_0(\Om)\cap H^2(\Om)$.  Then, if $k<1$ and we are in the connected/connected case, the unique solution $u_\eps$ of the problem \eqref{eq:PDEink}--\eqref{eq:InitDatak} converges to $0$, weakly in $L^2\big(0,T;\X^\eps_0(\Om)\big)$,
for $\eps\to 0$.
\end{thm}

\begin{thm}\label{t:casoKpiccolo_bis}
Assume that $f\in L^2(\Om)$ and $ \overline u_{0\eps} = \eps^{(1-k)/2}\overline u_0$, with $\overline u_0\in
H^1_0(\Om)\cap H^2(\Om)$.  Then, if $k<1$ and we are in the connected/disconnected case, the unique solution $u_\eps$ of the problem \eqref{eq:PDEink}--\eqref{eq:InitDatak} weakly converges to $u\in L^2\big(0,T;H^1_0(\Om)\big)$,
for $\eps\to 0$, where $u$ is the unique solution of the problem
$$
\begin{aligned}
& -\!\Div\left(
{
\Big(\int_{\Perout}\!\!\lfout (I+\nabla_y\chi_0)\di y+\int_{\Permemb}\!\!\lfout \big((\nu+\nu\nabla_y\chi_0^{\rm{out}})\otimes y_M\big)\di\sigma
\Big)
}
\nabla u\right)\!\!=f,\ \ \text{in $\Om_T$}\,;&\\
& u=0,\qquad\qquad\qquad\qquad\qquad\qquad\qquad\qquad\qquad\qquad\qquad\qquad\qquad\ \text{on $\partial\Om\times(0,T)$},&
\end{aligned}
$$
where $\chi^j_0\in H^1_{\#}(Y)$, $j=1,\dots,N$, are $Y$-periodic functions with null mean average satisfying
\eqref{cell-chi0_1}--\eqref{cell-chi0_3} and $y_M=y-\int_Y y\di y$.
The homogenized matrix $A^{\rm{hom}}:=\int_{\Perout}\lfout (I+\nabla_y\chi_0)\di y+\int_{\Permemb}\lfout \big((\nu+\nu\nabla_y\chi_0^{\rm{out}})
\otimes y_M\big)\di\sigma$
is symmetric and positive definite.
\end{thm}

\begin{thm}\label{t:casoKgrande}
Assume that $f\in L^2(\Om)$ and $ \overline u_{0\eps} = \eps^{(1-k)/2}\overline u_0$, with $\overline u_0\in
H^1_0(\Om)\cap H^2(\Om)$.  Then, if $k>1$, the unique solution $u_\eps$ of the problem \eqref{eq:PDEink}--\eqref{eq:InitDatak} weakly converges to $u\in L^2\big(0,T;H^1_0(\Om)\big)$,
for $\eps\to 0$, where $u$ is the unique solution of the problem
$$
\begin{aligned}
& -\Div\left(\Big(\int_{Y}\lfboth(I+\nabla_y\widetilde\chi_0)\di y\Big)\nabla u\right)=f\,,\qquad
&\text{in $\Om_T$}\,;\\
& u=0\,,\qquad &\text{in $\partial\Om\times(0,T)$}\,,
\end{aligned}
$$
where $\widetilde\chi^j_0\in H^1_{\#}(Y)$, $j=1,\dots,N$, are $Y$-periodic functions with
null mean average satisfying \eqref{eq:a33}--\eqref{eq:a34}.
The homogenized matrix $A^{\rm{hom}}:=\int_{Y}\lfboth(I+\nabla_y\widetilde\chi_0)\di y$ is symmetric and positive definite.
\end{thm}

\section{Derivation of the concentrated problem}
\label{s:concentrazione}
In this section, using a local parametrization of the regular surface $\Memb$
as in \cite[Section 3]{Amar:Gianni:2016C} (see also \cite[Section 3]{Amar:Andreucci:Bisegna:Gianni:2006a})
and following the outline of \cite[Theorem 3.1]{Amar:Gianni:2016C}, we prove the following result.

\begin{thm}\label{t:concentration}
Assume that $\overline u_0\in W^{1,\infty}_0(\Om)\cap H^2(\Om)$.
Let $u_\eps^\eta$ be the solution of \eqref{eq:a14}. Then,
\begin{equation}\label{eq:a20}
\soluz \rightharpoonup u_{\eps}\,, \quad
\nabla\soluz \rightharpoonup\nabla u_{\eps}\,,
\qquad \text{weakly in
  $L^{2}(\Om_T)$,}
\end{equation}
where $u_\eps$ is the solution of \eqref{eq:weak_sol}.
\end{thm}

\begin{proof}
We first note that, as proven in \cite[Proposition 2.2]{Amar:Andreucci:Gianni:Timofte:2017B},
one can derive from \eqref{eq:a14} the energy inequality
\begin{equation}
  \label{eq:energy_a}
\begin{aligned}
& \int_{0}^{T}\int_{\nOmoutdue\cup\nOmoutuno} \abs{\nabla \soluz}^{2}
  \di x\di t
  +  \frac{1}{\eta}\sup_{t\in(0,T)}\int_{\nOmint} \abs{\nabla \soluz}^{2}(t) \di x
 \\
\le &\frac{\alpha}{\eta \min(\lfint,\lfout,\alpha)}\int_{\nOmint} \abs{\nabla \overline u_0}^{2}\di x
\le\gamma \frac{\alpha\,\eta|\Memb|}{\eta \min(\lfint,\lfout,\alpha)}
\Vert{\nabla \overline u_0}\Vert^{2}_{L^\infty(\Om)}\le \gamma\,,
\end{aligned}
\end{equation}
where $\gamma$ depends on $\eps,\lfint,\lfout,\alpha,
\Vert \nabla \overline u_0\Vert_{L^{\infty}(\Om)}$,
but not on $\eta$.
As a consequence of \eqref{eq:energy_a}, as $\eta\to 0$, we may assume, extracting a subsequence if needed, that
\begin{equation*}
\soluz \rightharpoonup u_{\eps}\,, \quad
\nabla\soluz \rightharpoonup\nabla u_{\eps}\,,
\qquad \text{weakly in
  $L^{2}(\Om_T)$,}
\end{equation*}
where, for every $\eps>0$, $u_\eps\in L^2\big(0,T;H^1_0(\Om)\big)$.

In order to proceed with the concentration of problem \eqref{eq1}--\eqref{eq5}, we need to choose a suitable testing function
in the weak formulation \eqref{eq:a14}, before passing to the limit for $\eta\to0$.
To this purpose, we recall that there exists an $\eta_0>0$, such that
for $\eta<\eta_0$, the application
$$
\psi:\Memb\times[-\eps\eta,\eps\eta]\to \wnOmint\,,\qquad
\psi(\yg,r)=\yg+r\nu_\eps(\yg)=y\in\wnOmint
$$
is a diffeomorfism onto its image, where we denote by
$\wnOmint$ the tubular neighborhood of $\Memb$ with thickness $2\eps\eta$.
Clearly, $\wnOmint$ can be considered as the union of surfaces denoted by $\Memb_r$
parallel
to $\Memb$ and at distance $|r|$ from it, when $r$ varies in $[-\eps\eta,\eps\eta]$.
Hence, for $y\in\wnOmint$, there exists a unique $(\yg,r)\in\Memb\times[-\eps\eta,\eps\eta]$
such that $y=\yg+r\nu_\eps(\yg)$ and, then, $y\in \Memb_r$ and $\nu_\eps(\yg)$ coincides with
the normal to the surface $\Memb_r$ at $y$.
Moreover, we can locally parametrize $\Memb$ in such a way that there exist $\wMemb\subset\R^{N-1}$
and $ \yg:\wMemb \to \Memb$ such that
$\Memb\ni \yg=\yg(\xi)$, where $\xi=(\xi_1,\dots,\xi_{N-1}) \in\wMemb$ and,
if we set $\di\sigma=\sqrt{g(\xi)}\di\xi$, we may assume that $\gamma_1\leq \sqrt{g(\xi)}
\leq\gamma_2$, for every $\xi\in \wMemb$, where $\gamma_1,\gamma_2$ are suitable strictly positive
constants.
As a consequence, we have obtained a change of coordinates in $\R^N$,
whose Jacobian matrix will be denoted by $J(\xi,r)$, defined by
$$
\wnOmint\ni y =(y_1,\dots,y_N)\longleftrightarrow (\xi,r)=(\xi_1,\dots,\xi_{N-1},r)
\in \wMemb\times[-\eps\eta,\eps\eta]\,.
$$
By the assumed regularity of $\Memb$, it follows
that $J(\xi,r)=J(\xi,0)+M_\eta$, where $M_\eta$ denotes a suitable matrix such that
$|M_\eta|\leq\gamma\eta$, so that $|det J(\xi,r)| =|det J(\xi,0)|+R_\eta$, where
$|R_\eta|\leq\gamma\eta$; moreover, by the choice of the coordinates $(\xi,r)$, we have that $|det J(\xi,0)|=\sqrt{g(\xi)}$ (recall that the volume element
$\di y=|det J(\xi,r)|\di\xi\di r$, for $r=0$, i.e. on $\Memb$, becomes
$\di y=|det J(\xi,0)|\di\xi\di r=\di\sigma\di r=\sqrt{g(\xi)}\di\xi\di r$).

Finally, we define $\pi_0(y)=\yg$ as the orthogonal projection
of $y\in\wnOmint$ on $\Memb$ and
$\rho(y)=r$ as the signed distance of $y\in\wnOmint$ from $\Memb$. Note that
$|\nabla\rho(y)|$ is bounded.

In the sequel, we assume without
loss of generality that the support of our testing functions is sufficiently
small to allow for the representation introduced above.
The general case can then be recovered by means of a standard partition of unity
argument.
Moreover, for the sake of brevity, we will use the same symbol for the same function
even if written with respect to different variables.

Let now $\Phi\in \CC^{\infty}(\overline\Om_T)$
be any given testing function for the concentrated problem \eqref{eq:PDEin}--\eqref{eq:InitData}
as in Definition \ref{d:weak_sol}, with the additional assumption on its support stated above.
Starting from $\Phi$, we construct a suitable test function $\testeta$ for problem \eqref{eq:a14}
in such a way that it does not depend on the transversal coordinate
inside $\nOmint$ (thus it is equal to its value on $\Memb$) and it is linearly connected
with $\Phi$ in $\nOmoutuno$ and $\nOmoutdue$ along the $r$-direction.
It is crucial in order to develop the concentration procedure to make this
gluing
where the diffusivity in equation \eqref{eq1} is stable with respect to $\eta$, i.e.
inside the set $\wnOmint\setminus\nOmint\subset\nOmoutuno\cup\nOmoutdue$.
To this purpose, define
\begin{equation}\label{eq:a15}
\testeta(y,t) = \left\{
\begin{aligned}
& \Phi(y,t) \qquad & \text{if $(y,t)\in (\nOmoutdue\setminus\wnOmint)\times(0,T)$;}
\\
& \testeta_{\text{out}}(y,t)
\qquad & \text{if $(y,t)\in(\nOmoutdue\cap\wnOmint)\times(0,T)$;}
\\
& \Phi\big(\pi_{0}(y),t\big)
\qquad & \text{if $(y,t)\in\nOmint\times(0,T)$;}
\\
& \testeta_{\text{int}}(y,t)
\qquad & \text{if $(y,t)\in(\nOmoutuno\cap\wnOmint)\times(0,T)$;}
\\
& \Phi(y,t) \qquad & \text{if $(y,t)\in (\nOmoutuno\setminus\wnOmint)\times(0,T)$,}
\end{aligned}
\right.
\end{equation}
where
\begin{multline*}
\testeta_{\text{out}}(y,t)=\bigg[ \Phi\big(\pi_{0}(y)+\eps\eta\nu_\eps(\pi_0(y)),t\big) -
  \Phi\big( \pi_{0}(y),t\big) \bigg]
  \frac{2\rho(y)-\eps\eta}{\eps\eta}
  +\Phi\big( \pi_{0}(y),t\big)
\end{multline*}
and
\begin{multline*}
\testeta_{\text{int}}(y,t)
= \bigg[  \Phi\big( \pi_{0}(y),t\big) -
\Phi\big( \pi_{0}(y)-\eps\eta \nu_\eps(\pi_0(y)),t\big)\bigg]
  \frac{2\rho(y)+\eps\eta}{\eps\eta}
   +\Phi\big( \pi_{0}(y),t\big).
\end{multline*}
By a density argument, we can use the Lipschitz continuous function $\testeta$ as a testing function in
\eqref{eq:a14}; then, it follows that
\begin{multline}\label{eq:a16}
 \int_{0}^{T}\!\!\int_{(\nOmoutuno\cup\nOmoutdue)\setminus\wnOmint} \lfbothe \nabla \soluz\cdot\nabla\Phi \di y\di t
 -\frac{\alpha}{\eta}\int_{0}^{T}\!\!\int_{\nOmint} \nabla \soluz\cdot\nabla
 \Big(\Phi_t\big( \pi_{0}(y),t\big)\Big)\di y\di t
  \\
 +\int_{0}^{T}\!\!\int_{\nOmoutuno\cap\wnOmint} \lfbothe \nabla \soluz\cdot\nabla\testeta_{\text{int}}\di y\di t
 +\int_{0}^{T}\!\!\int_{\nOmoutdue\cap\wnOmint} \lfbothe \nabla \soluz\cdot\nabla\testeta_{\text{out}}\di y\di t
 \\
 =
\frac{\alpha}{\eta}\int_{\nOmint} \nabla\overline u_0\cdot
\nabla\Big(\Phi\big( \pi_{0}(y),0\big)\Big)\di y\,.
 \end{multline}
We take into account that
\begin{equation*}
\nabla \testeta_{\text{out}}(y,t)=\Im(\eta)
+\bigg[\Phi\big( \pi_{0}(y)+\eps\eta\nu_\eps(\pi_0(y)),t\big) -
  \Phi\big( \pi_{0}(y),t\big) \bigg]
  \frac{2\nabla\rho(y)}{\eps\eta}\,,
\end{equation*}
where with $\Im(\eta)$ we denote a bounded quantity with respect to $\eta$, and that
\begin{equation}
\Big\vert\Phi\big( \pi_{0}(y)+\eps\eta\nu_\eps(\pi_0(y)),t\big) -
  \Phi\big( \pi_{0}(y),t\big)\Big\vert\le \gamma\eps\eta\,,
  \end{equation}
with $\const$ independent of $\eta$. Clearly, similar estimates hold for $\testeta_{\text{int}}$.
Owing to \eqref{eq:energy_a}, it is easy to see that, when $\eta\to 0$,
the second line in the equality \eqref{eq:a16} tends to $0$. In addition obviously
\begin{equation*}
\int_{0}^{T}\!\!\int_{(\nOmoutuno\cup\nOmoutdue)\setminus\wnOmint} \lfbothe \nabla \soluz\cdot\nabla\Phi \di y\di t
\to  \int_{0}^{T}\!\!\int_{\Omint\cup\Omout} \lfbothe \nabla u_\eps\cdot\nabla\Phi \di y\di t\,.
\end{equation*}
Hence, the crucial limits are the second and the fifth ones in \eqref{eq:a16}. Let us deal with
the second limit; the fifth can be treated in a similar and even simpler way. In order to do this,
we pass to the new coordinates $(\xi,r)$ defined above,
recalling that $J(\xi,r)$ denotes the Jacobian matrix of such a change of coordinates.
Moreover, denoting by $\tangrad{r}$ the tangential gradient with respect to the surface $\Memb_r$ and
recalling that the normal vector at $y\in\Memb_r$ coincides with the normal at $\pi_0(y)\in\Memb$,
we have $\tangrad{r} \soluz=\nabla\soluz-(\nu_\eps(\pi_0(y))\cdot \nabla\soluz) \nu_\eps(\pi_0(y))$,
with $r=\rho(y)$.
Also, since the test function does not depend on the normal coordinate $r$ in $\nOmint$, we have
that $\nabla\big(\Phi(\pi_0(y),t)\big)=\beltramigrad\big(\Phi(\pi_0(y),t)\big)$ and hence
$\nabla\soluz\cdot\nabla\Phi=\tangrad{\rho(y)}\soluz\cdot\beltramigrad\Phi$.

Next we denote by $\widetilde J(\xi,r)$ the $N\times(N-1)$ rectangular matrix such that,
for every function $v(y)$,
$\widetilde J(\xi,r)\nabla_\xi v(\xi,r)=\tangrad{\rho(y)} v(y)$,
and we let, for the sake of simplicity, $\widetilde J(\xi):=\widetilde J(\xi,0)$.
Then, we can rewrite
\begin{multline*}
\frac{\alpha}{\eta} \int_0^T\int_{\nOmint}\nabla\soluz\cdot
\nabla\bigg(\Phi_t\big(\pi_0(y),t\big)\bigg)
\di y\di t
=\frac{\alpha}{\eta} \int_0^T\int_{\nOmint}\tangrad{\rho(y)}\Phi_t\big(\pi_0(y),t\big)
\cdot \tangrad{\rho(y)}\soluz
\di y\di t
\\
=\frac{\alpha}{\eta} \int_0^T\int_{\wMemb}\int_{-\eps\eta/2}^{\eps\eta/2}
(\widetilde J(\xi,r)\nabla_{\xi}\Phi_t(\xi,0,t))^T
\widetilde J(\xi,r)\nabla_\xi\soluz|det J(\xi,r)| \di \xi\di r\di t
\\
= {\alpha\eps} \int_0^T\int_{\wMemb}
\Big(\widetilde J(\xi)\nabla_\xi\Phi_t\big(\xi,0,t\big)\Big)^T
\left(\frac{1}{\eps\eta} \int_{-\eps\eta/2}^{\eps\eta/2}
\widetilde J(\xi)\nabla_\xi\soluz\big(\xi,r,t\big)
\di r\right)\sqrt{g(\xi)}\di \xi\di t+I_2(\eta)
\\
=: I_1+I_2(\eta)\,,
\end{multline*}
where the superscript $^T$ denotes the transposed vector.
Obviously, due to the regularity of $\Memb$, also the matrix $\widetilde J$ is regular, so that
$\widetilde J(\xi,r)=\widetilde J(\xi,0)+O(\eta)$.

Clearly, using the energy estimate \eqref{eq:energy_a}, we obtain
\begin{equation*}
|I_2(\eta)|
\le  \gamma \frac{\eta}{\sqrt\eta}\left( \frac{1}{\eta}\int_0^T\int_{\nOmint}|\nabla\soluz|^2\di y
\di t\right)^{1/2}\sqrt\eta\le \gamma\eta\to 0\qquad\hbox{as $\eta\to 0$.}
\end{equation*}
On the other hand, again
by the energy estimate \eqref{eq:energy_a}, it follows that there exists a vector function
${\bf V}\in L^2(0,T;L^2(\wMemb))$ such that, up to a subsequence,
$$
\frac{1}{\eps\eta} \int_{-\eps\eta/2}^{\eps\eta/2}\widetilde J(\xi)\nabla_\xi\soluz(\xi,r,t)
\di r\rightharpoonup {\bf V}\,,
\qquad \hbox{weakly in $L^2(0,T;L^2(\wMemb))$,}
$$
so that
\begin{equation*}
I_1
\to {\alpha\eps} \int_0^T\int_{\wMemb}\Big(\widetilde J(\xi)\nabla_\xi\Phi_t(\xi,0,t)\Big)^T
{\bf V}\sqrt{g(\xi)}\di \xi\di t
={\alpha\eps} \int_0^T\int_{\Memb}\beltramigrad\Phi_t
\cdot{\bf V}\di \sigma\di t
\,.
\end{equation*}
It remains to identify ${\bf V}$ as the tangential gradient of the limit $u_\eps$; i.e.,
${\bf V} = \beltramigrad u_\eps$ on $\Memb$. To this aim, we consider a vector test function $\Psi\in\CC^1_c(\Om_T)$; we obtain
\begin{multline*}
\int_0^T\int_{\Memb}\beltramidiv\Psi\, u_\eps \di\sigma\di t
\longleftarrow
\int_0^T\int_{\Memb} \beltramidiv\Psi\left( \frac{1}{\eps\eta}\int_{-\eps\eta/2}^{\eps\eta/2}
\soluz(\yg+r\nu_\eps(\yg),t)\di r\right)\di \sigma\di t
\\
=-\int_0^T\int_{\Memb} \Psi \cdot\beltramigrad \left(\frac{1}{\eps\eta}\int_{-\eps\eta/2}^{\eps\eta/2}
\soluz(\yg+r\nu_\eps(\yg),t)\di r\right)\di \sigma\di t
\\
=-\int_0^T\int_{\wMemb} \Psi \cdot\left(\frac{1}{\eps\eta}\int_{-\eps\eta/2}^{\eps\eta/2}
\widetilde J(\xi)\nabla_\xi\soluz(\xi,r,t)\di r\right)\sqrt{g(\xi)}\di \xi\di t
\longrightarrow
\\
-\int_0^T\int_{\wMemb} \Psi \cdot{\bf V}\sqrt{g(\xi)}\di\xi\di t=
-\int_0^T\int_{\Memb} \Psi \cdot{\bf V}\di\sigma\di t\,,
 \end{multline*}
which implies that ${\bf V}=\beltramigrad u_\eps$.
Similarly, we can prove that
\begin{equation*}
\frac{\alpha}{\eta} \int_{\nOmint}\nabla\overline u_0\cdot
\nabla\bigg(\Phi\big(\pi_0(y),0\big)\bigg)
\di y\to
\alpha\eps \int_{\Memb}\beltramigrad\overline u_0
\cdot \beltramigrad\Phi(y,0) \di \sigma\,.
\end{equation*}
This proves that the limit for $\eta\to 0$ of equality \eqref{eq:a16}
yields \eqref{eq:weak_sol}; i.e., the concentration limit of $\soluz$ is the weak solution of system
\eqref{eq:PDEin}--\eqref{eq:InitData}.

By the uniqueness of solutions to the limit problem \eqref{eq:weak_sol}, the whole sequence $\{\soluz\}$ converges.
\end{proof}

\section{Homogenization of the microscopic problem}
\label{s:homog}

Our goal in this section is to describe the asymptotic behavior, as $\eps\to 0$, of the solution $u_\eps\in \spaziosoleps$ of problem \eqref{eq:PDEin}--\eqref{eq:InitData}.
\smallskip

From \cite[Theorem 2.8]{Amar:Andreucci:Gianni:Timofte:2017B}, we obtain the following result.

\begin{thm}\label{aprioriest}
If $\overline u_0\in H^1_0(\Om)\cap H^2(\Om)$,
then, for any $\eps\in(0,1)$, the variational problem $\eqref{eq:weak_sol}$ has a unique solution $u_\eps\in \spaziosoleps$. Moreover, there exists a constant $\gamma>0$,
  independent of $\varepsilon$, such that
\begin{equation}
\label{eq:energy}
\int_{0}^{T}\!\!\int_{\Om} \abs{\nabla u_\eps}^{2}
  \di x\di\tau
  +  \eps\sup_{t\in(0,T)}\int_{\Gamma^\eps} \abs{\nabla^B u_\eps}^{2} \di \sigma \le \gamma.
\end{equation}
\end{thm}

\begin{remark}\label{r:r9}
Notice that, as proven in \cite[Section 3]{Amar:Gianni:2016A}, the assumption
$\overline u_0\in H^1_0(\Om)\cap H^2(\Om)$
implies that, for every $\eps>0$, $\overline u_0\in H^1(\Memb)$ and
\begin{equation}\label{eq:a3}
\eps\int_{\Memb}|\beltramigrad \overline u_0|^2\di \sigma\leq \gamma\,,
\end{equation}
with $\gamma$ independent of $\eps$. This is the crucial tool in order to obtain \eqref{eq:energy}
above.
\end{remark}

The convergence results
stated in the next lemma are a consequence of the {\it a priori} estimates \eqref{eq:energy} and of the general compactness results obtained in Subsection \ref{ss:unfold} (see Theorem \ref{t:smalleps_grad_weak_conv} and Proposition \ref{p:p1}).

\begin{lemma} \label{conv}
 Let $u_\eps \in \spaziosoleps$ be the unique solution of problem $\eqref{eq:weak_sol}$. Then, up to a subsequence,
  still denoted by   $\eps$, there exist  $u\in L^2(0, T; H^1_0(\Om))$ and $w\in L^2(\Om_T; \perzero)$ with
   ${\mathcal{M}}_Y(w)=0$ and $\beltramigrad_y w \in L^2(\Om_T \times \Gamma)$ such that
 \begin{equation*}
  \begin{aligned}
    & u_\eps \rightharpoonup u \quad &\text{ weakly in $ L^2(0,T;H^1_0(\Om))$,}
    \\
      & \unfop (u_\eps)  \rightharpoonup u \quad &\text{ weakly in  $L^2\left(\Omega_T \times Y\right)$,}
      \\
   &  \unfop (\nabla u_\eps) \rightharpoonup \nabla u+\nabla_y w \quad &\text{ weakly in
   $L^2(\Omega_T \times Y)$,}
   \\
   &  \btsunfop (\beltramigrad u_\eps) \rightharpoonup \beltramigrad u+\beltramigrad_y w  \quad &\text{ weakly in $L^2(\Omega_T \times \Permemb)$.}
       \end{aligned}
       \end{equation*}
\end{lemma}

\begin{thm}\label{th-conv-unfold}
Let us set
\begin{equation*}
\begin{aligned}
{\mathcal W}_{\#}(Y) & =\{v\in H^1_{\#}(Y) \, \vert \,  {\mathcal M}_Y (v) = 0, \, \nabla_y^B v \in L^2(\Om_T \times \Permemb)\}
\\
{\mathcal V}&=L^2(0,T; H^1_0(\Om)) \times L^2(\Om_T; {\mathcal W}_{\#}(Y)).
\end{aligned}
\end{equation*}
The unique solution $u_\eps\in \spaziosoleps$ of the variational problem
$\eqref{eq:weak_sol}$ converges,
in the sense of Lemma \ref{conv}, to the unique solution $(u,w)\in {\mathcal{V}}$
of the following unfolded limit problem
\begin{equation} \label{unfold}
\begin{aligned}
& \displaystyle \int_0^ T\!\!\int_{\Omega\times Y}{} \lambda \, (\nabla u +\nabla_y w) \cdot
     (\nabla \varphi + \nabla_y \Psi)  \di x \di y \di t
 \\
   - \alpha & \displaystyle \int_0^T \!\!\int_{\Omega\times \Gamma}{} (\beltramigrad u+\beltramigrad_y w) \cdot (\beltramigrad \varphi_t +\beltramigrad_y \Psi_t )\di x \di \sigma \di t
     \\
     = \alpha &\displaystyle \int_{\Om \times \Gamma} {} \beltramigrad \overline{u}_0 \cdot (\beltramigrad \varphi (x,0) + \beltramigrad_y \Psi (x,y,0))\di x \di \sigma,
                  \end{aligned}
         \end{equation}
 for all $\varphi\in H^1 \big(0,T;H^1_0(\Om)\big)$ and $\Psi \in L^2(\Om_T; H^1_{\#}(Y))\cap
 H^1\left( 0,T; L^2(\Om, H^1(\Permemb))\right)$ with $\varphi(\cdot,T) = 0$ in $\Om$, $\Psi (\cdot, \cdot,T) = 0$ in $\Om \times Y$.
\end{thm}

\begin{proof}
Preliminarily,
we note that equation \eqref{unfold} admits at most one solution. Indeed, setting $U=u_2-u_1$ and
$\WW=w_2-w_1$, where $(u_i,w_i)$, $i=1,2$, are two solutions of \eqref{unfold},
we obtain in a standard way from the equation written for $(U,\WW)$
\begin{equation*}
 \displaystyle \int_0^ T\!\!\int_{\Omega\times Y}{} \lambda |\nabla U +\nabla_y \WW|^2\di x \di y \di t
     +\frac{\alpha}{2} \displaystyle \int_{\Omega\times \Gamma}{}
     |\beltramigrad U+\beltramigrad_y \WW|^2(T)\di x \di \sigma =0\,.
         \end{equation*}
Hence, dropping the last integral (which is nonnegative) and using the $Y$-periodicity of $\WW$, we obtain
\begin{multline*}
 \int_0^ T\!\!\int_{\Omega\times Y} (|\nabla U|^2 +|\nabla_y \WW|^2)\di x \di y \di t
\\
= \int_0^ T\!\!\int_{\Omega\times Y} (|\nabla U|^2 +|\nabla_y \WW|^2)\di x \di y \di t
+ 2\int_0^ T\!\!\int_{\Omega} \left(\nabla U\cdot\int_{Y}\nabla_y \WW\di y\right) \di x \di t
\\
=\int_0^ T\!\!\int_{\Omega\times Y}|\nabla U +\nabla_y \WW|^2\di x \di y \di t
\leq0\,.
 \end{multline*}
 Therefore, taking into account that $U$ vanishes on $\partial\Om$ and $\WW$ has null mean average in $Y$,
 we get $U=\WW=0$; i.e., the asserted uniqueness.

In order to obtain the limit problem (\ref{unfold}), we choose in the variational formulation
 (\ref{eq:weak_sol}) the  admissible test function
 \begin{equation}\label{testfct}
 \Phi (x,t) =\varphi(x,t)+\varepsilon \phi (x,t) \psi \left( \displaystyle \frac{x}{\varepsilon}\right),
 \end{equation}
with $\varphi, \phi \in \CC^\infty([0,T]; \CC^\infty_c(\Om))$, $\varphi(\cdot,T) = \phi(\cdot,T) = 0$ in $\Om$ and $\psi\in \CC^\infty_{\#}(Y)$.

Then, by unfolding each term with the corresponding operator, we get
\[
 \displaystyle \int_{0}^{T}\!\!\int_{\Om\times Y}\unfop(\lfbothe) \unfop( \nabla u_{\eps})\cdot \unfop(\nabla\Phi) \di y \di x\di t
  -\alpha \int_0^T \!\!\int_{\Om\times\Permemb}\btsunfop (\beltramigrad u_{\eps}) \cdot \btsunfop (\beltramigrad \Phi_t) \di\sigma \di x \di t
  \]
  \begin{equation}\label{unfold:weak-sol}
 ={\alpha}\!\!\int_{\Om\times\Permemb} \btsunfop(\beltramigrad\overline u_0) \cdot \btsunfop(\beltramigrad\Phi(x,0))\di\sigma \di x  +R_\eps,
\end{equation}
where $R_\eps=o(1)$ as $\eps \rightarrow 0$. Our goal now is to pass to the limit with $\eps\to 0$ in \eqref{unfold:weak-sol}. By Remark \ref{r:r6}, Proposition \ref{c:c1} and
Proposition \ref{p:per_odc_fun}, we obviously  have
\begin{alignat}2
\label{conv-grad-Phi}
&\unfop (\nabla \Phi) \rightarrow \nabla \varphi + \nabla_y \Psi  \quad &\text{strongly in
$L^2(\Om_T \times Y)$;}
\\
\label{conv-beltramigrad-Phi}
&\btsunfop (\beltramigrad \Phi_t) \rightarrow \beltramigrad \varphi_t + \beltramigrad_y \Psi_t
\quad
&\text{strongly in $L^2(\Om_T \times \Permemb)$;}
\\
\label{conv-beltramigrad-Phi0}
&\btsunfop (\beltramigrad \Phi(x,0)) \rightarrow \beltramigrad \varphi(x,0)
+ \beltramigrad_y \Psi(x,0)\quad
&\text{strongly in $L^2(\Om_T \times \Permemb)$,}
\end{alignat}
where $\Phi$ is given in \eqref{testfct} and $\Psi (x,y,t)=\phi(x,t) \psi(y)$.
Therefore, by using the convergence results stated in Lemma \ref{conv} and recalling again \eqref{eq:per_osc_fun_ii} and
Proposition \ref{c:c1}, we obtain
\begin{multline*}
\int_0^ T \!\!\int_{\Omega\times Y}{} \lambda \, (\nabla u +\nabla_y w) \cdot
(\nabla \varphi + \nabla_y \Psi) \di x \di y \di t
\\
-\alpha  \int_0^T \!\!\int_{\Omega\times \Gamma}{} (\beltramigrad u+\beltramigrad_y w) \cdot (\beltramigrad \varphi_t +\beltramigrad_y \Psi_t )\di x \di \sigma \di t=
\\
   \alpha\int_{\Om \times \Gamma} {} \beltramigrad \overline{u}_0 \cdot (\beltramigrad \varphi (x,0) + \beltramigrad_y \Psi (x,y,0))\di x \di \sigma.
\end{multline*}
Standard density arguments lead us to (\ref{unfold}).
Moreover,
due to the uniqueness of $(u,w)\in {\mathcal{V}}$, all the above convergences hold true for
the whole sequence.
\end{proof}

\begin{remark}\label{r:r1}
Our goal now is to obtain the factorized formulation of the unfolded problem (\ref{unfold}). Let us point out that this formulation involves the time derivatives of the solution and of its corrector (in the distributional sense). Due to the presence of these terms generated by the dynamical boundary condition on $\Gamma^\eps$ in the microscopic problem, we have to introduce a non-standard type of cell functions containing memory terms. Our limit model can be compared with \cite{Amar:Andreucci:Bisegna:Gianni:2003a,Amar:Andreucci:Bisegna:Gianni:2004a,
Amar:Andreucci:Bisegna:Gianni:2010,Amar:Gianni:2016A}.
 \end{remark}

{\it Proof of Theorem \ref{strong-form}}
By taking first $\varphi=0$ and suitable test functions $\Psi$, in  the unfolded limit problem (\ref{unfold}), and then $\Psi=0$ with suitable choices of $\varphi$, we obtain
formally
\begin{alignat} 2
\label{homog-1_1}
&-\Div\!\Big (\!\lfav\!\nabla u\! +\!\! \displaystyle \int_Y\! \!\lfboth \nabla_y w \di y \!
\Big )\!\!-\!
\Div\!\Big  (\alpha\Big (\!\int_\Permemb \! \beltramigrad  u_t \di \sigma \!+\!\!
\int_\Permemb \!\!\beltramigrad_y w_t \di \sigma \Big)\! \Big )\!\! = \!0,
\!\!\! \!\!\!\!\!\!\!\!\!\!\!\!\!\!\!\!\!\!\!\!\!\!\!\!\!\!\!\!\!\!\!\!\!\!\!\!\!\!
&\text{in $\Omega_T$;}
\\
\label{homog-1_2}&-\Div_y \left (\lfboth (\nabla u + \nabla_y w ) \right )= 0 \,,
&\text{in $\Om_T \times (\Perint\cup \Perout)$;}
\\
\label{homog-1_3}&- \beltramidiv _y  \left (\alpha (\beltramigrad  u_t + \beltramigrad_y w_t) \right )=
\left [\lfboth ( \nabla u + \nabla_y w)\cdot \nu\right ] \,,
&\text{on $\Om_T \times \Permemb$;}
\\
\label{homog-1_4}& u=0\,,  &\text{on $\partial \Omega \times (0,T)$;}
\end{alignat}
complemented with the initial conditions
\begin{alignat} 2
\label{homog-1_6}& -\Div\left(\int_{\Permemb}(\beltramigrad u(x,0)+\beltramigrad_y w(x,y,0))
\di \sigma\right)=
-\Div\left(\int_{\Permemb}\beltramigrad \overline{u}_0\di \sigma\right)\!,
\!\!\!\!\!\!\!\!\!\!\!&\text{in $\Omega$.}
\\
\label{homog-1_5}& -\beltramidiv_y\left(\beltramigrad u(x,0)+\beltramigrad_y w(x,y,0)\right)=
-\beltramidiv_y\left(\beltramigrad \overline{u}_0\right)\!,
\!\!\!\!\!\!\!\!\!\!\!\!\!\!\!\!\!\!\!\!\!\!\!\! &\text{on $\Omega \times\Permemb$,}
\end{alignat}
Note that \eqref{homog-1_6} is the initial condition associated to \eqref{homog-1_1}, while
\eqref{homog-1_5} is the initial condition associated to \eqref{homog-1_3}.

Equation \eqref{homog-1_5}
for $w$ is uniquely solvable on each connected component $\Permemb_i$, $i=1,\dots,m$, of $\Permemb$
in terms of $u$ and $\overline u_0$, up to an additive $y$-constant function $C_i(x)$, depending on the
$i^{th}$ connected component. The functions $C_i(x)$ can be chosen identically equal to $0$,
since they do not play any role in \eqref{homog-1_5}. Moreover, it is easy to prove that the function
$w(x,y,0)=-\chi_0(y)\cdot[\nabla\overline u_0(x)-\nabla u(x,0)]$, with $\chi_0$ defined
by \eqref{cell-chi0_1}--\eqref{cell-chi0_5} below, is the required solution.
In particular, if $\Omout$ is connected and $\Omint$ is disconnected, then $\chi_0(y)=-y+c_i$, on each
connected component $\Permemb_i$, so that
\begin{equation}\label{homog-1_9}
\beltramigrad_y w(x,y,0)= \beltramigrad_y\left(y\cdot[\nabla\overline u_0(x)-\nabla u(x,0)]\right)
=\beltramigrad\overline u_0(x)-\beltramigrad u(x,0)
\,,
\end{equation}
and equation \eqref{homog-1_6} becomes simply an identity.
On the contrary, if $\Omout$ and $\Omint$ are both connected, then \eqref{homog-1_6} becomes
\begin{equation}\label{homog-1_10}
-\Div\big(C^0\nabla u(x,0)\big)= -\Div\big(C^0\nabla\overline u_0\big)\,,
\end{equation}
where $C^0$ is the matrix defined in \eqref{matrix-C} below.
%
%

The presence of the time derivatives in \eqref{homog-1_1} and  \eqref{homog-1_3}
suggests us looking for the corrector $w$ in the non-standard form given in \eqref{corrector}.
The factorization in terms of the cell function $\chi_0: Y \rightarrow \Bbb R^N$
is rather standard, though the problem which defines $\chi_0$ is not
(see \eqref{cell-chi0_1}--\eqref{cell-chi0_3}).
However, due to the dynamical boundary condition on $\Gamma^\eps$, apart from this function, we need to introduce two new cell functions, $\chi_1: Y \times (0,T) \rightarrow \R^N$ and $W: \Om \times Y \times (0,T) \rightarrow \Bbb R$ (see \eqref{cell-chi1_1}--\eqref{cell-chi1_init} and \eqref{cell-W_1}--\eqref{cell-W_4}).

More precisely, introducing \eqref{corrector} in \eqref{homog-1_2}--\eqref{homog-1_5},
we are led to the following local problems for $\chi_0$ and, respectively, $\chi_1$:
\begin{alignat}2
\label{cell-chi0_1}
-\Div_y \, \left (\lambda \, \nabla_y (y_j+\chi_0^j ) \right )& = 0\,,  \qquad &\text{in $\Perint \cup \Perout$;}
\\
\label{cell-chi0_4}
[\chi_0^j]& =0\,,  \qquad &\text{on $\Permemb$;}
\\
\label{cell-chi0_2}
-\beltramidiv _y \, \left (\alpha  \, \beltramigrad_y (y_j +\chi_0^j) \right )& = 0\,, \qquad &\text{on $\Permemb$;}
\\
\label{cell-chi0_5}
\int_{\Permemb_i} \left(\nabla_y\chi^j_0\right)^{\text{(out)}}\cdot\nu \di \sigma& =0\,,\ i=1,\dots,m\quad &
\\
\label{cell-chi0_3}
\int_Y \chi^j_0 \di y& =0\,,\quad &
\end{alignat}
and
\begin{alignat}2
\label{cell-chi1_1}
-\Div_y \, \left (\lambda \, \nabla_y \chi_1^j \right )& = 0\,, \qquad &\text{in $ (\Perint\cup \Perout)\times (0,T)$;}
\\
\label{cell-chi1_5}
[\chi_1^j]& =0\,,  \qquad &\text{on $\Permemb\times(0,T)$;}
 \\
 \label{cell-chi1_2}
-\alpha \beltramidiv_y \, \left (\beltramigrad_y \chi^j_{1t} \right )&=\left [\lambda\nabla_y \chi_1^j \cdot \nu \right ]\,, \qquad
&\text{on $\Permemb\times (0,T)$;}
\\
 \label{cell-chi1_3}
\beltramigrad _y \chi^j_1 (y,0)&=\beltramigrad_y v_j(y)\,, \quad &\text{on $\Permemb$,}
\\
\label{cell-chi1_4}
\int_Y \chi^j_1 \di y& =0\,,\quad &
\end{alignat}
where the initial data $\beltramigrad_y v_j=\beltramigrad_y v_j(y)$ is the solution of the problem
\begin{equation}\label{cell-chi1_init}
- \alpha \beltramidiv _y \, \left (\beltramigrad_y v_j(y) \right )=\left [\lambda \nabla_y (y_j+\chi_0^j ) \cdot \nu \right ],\qquad \text{on $\Permemb=\bigcup_{i=1}^{m}\Permemb_i$.}
\end{equation}
The scalar function $W=W(x,y,t)$ is defined as the solution of the problem
\begin{alignat}2
\label{cell-W_1}
-\Div_y \, \left (\lambda \, \nabla_y W \right )&= 0\,, \qquad &\text{in $(\Perint\cup \Perout)\times (0,T) $;}
\\
\label{cell-W_5}
[W]& =0\,,  \qquad &\text{on $\Permemb\times(0,T)$;}
 \\
\label{cell-W_2}
-\alpha \beltramidiv _y \, \left (\beltramigrad _y W_t \right )&=\left [\lambda (\nabla_y W \cdot \nu) \right ]\,, \qquad
&\text{on  $\Permemb\times (0,T) $;}
\\
\label{cell-W_3}
\beltramigrad _y W(x,y, 0)&=-\beltramigrad_y \chi_0 (y)\nabla \overline{u}_0 (x)
\,,\qquad &
\text{on $\Permemb$;}
\\
\label{cell-W_4}
\int_Y W \di y&=0\,.&
\end{alignat}
The  cell problem \eqref{cell-chi0_1}--\eqref{cell-chi0_3} admits a unique
solution $\chi_0^j\in H^1_\#(Y)\cap H^1(\Permemb)$, for
$j=1,\dots,N$. Indeed, it can be solved starting from
\eqref{cell-chi0_2}, which gives the boundary conditions on each connected component $\Permemb_i$
of $\Permemb$,
up to an additive constant $c_i$ on $\Permemb_i$, $i=1,\dots,m$.
Then, we solve the problem \eqref{cell-chi0_1} in $\Perout$, and we use \cite[Proposition 2.6]{Amar:Andreucci:Gianni:Timofte:2017B} in order to satisfy also \eqref{cell-chi0_5},
properly choosing the constants $c_i$ appearing in the first step.
In the next step, we solve \eqref{cell-chi0_1}--\eqref{cell-chi0_4} in each connected
component of $\Perint$.
In the last step, we can add a global suitable constant to the solution thus obtained,
in order to satisfy also \eqref{cell-chi0_3}.
Notice that the conditions \eqref{cell-chi0_5} are crucial, since they are
required in order that the right-hand side of \eqref{cell-chi1_init}
satisfies the compatibility conditions needed for the existence of a unique solution
(up to an additive constant, which plays no role in \eqref{cell-chi1_3}, so that it can be chosen
equal to zero).
Finally, the cell problems \eqref{cell-chi1_1}--\eqref{cell-chi1_4} and
\eqref{cell-W_1}--\eqref{cell-W_4} have unique solutions belonging to
$L^2\big(0,T;H^1_\#(Y)\cap H^1(\Permemb)\big)$
by \cite[Theorem 2.8 and Remark 4.7]{Amar:Andreucci:Gianni:Timofte:2017B}.
Actually, by standard bootstrap arguments, it follows that $\chi_0\in \CC^\infty_\#(Y)$ and
$\chi_1\in\CC^\infty\big(0,T;\CC^\infty_\#(Y)\big)$.
By introducing the particular form of the corrector \eqref{corrector} in the equation \eqref{homog-1_1},
we get, rearranging the terms,
\begin{multline}\label{homog-2}
-\Div\Big  (\alpha\int_\Permemb  \beltramigrad  u_t \di \sigma +
\alpha\int_\Permemb \beltramigrad_y  \chi_0 \nabla u_t \di \sigma \Big)
\\
-\Div\Big ( \lfav \nabla u + \int_Y \lfboth \nabla_y \chi_0 \nabla u \di y
+\alpha\int_\Permemb \beltramigrad_y \chi_1(y,0) \nabla u(x, t) \di \sigma\Big)
\\
-\Div\Big (
\int_0^t\!\!\int_Y  \lfboth \nabla_y \chi_1 (y, t-\tau) \nabla u(x, \tau) \di \tau \di y
+\alpha \displaystyle \int_0^t\!\!\int_\Permemb \beltramigrad_y \chi_{1t} (y, t-\tau)\nabla u (x, \tau) \di \tau \di \sigma\Big )
\\
 =\Div \Big(
 \int_Y \lfboth \nabla_y W \di y +  \alpha\int_\Gamma \beltramigrad_y W_t \di \sigma \Big).
\end{multline}
The principal part of the equation in the first line of \eqref{homog-2} can be written as follows
\begin{multline}\label{eq:a2}
-\Div \Big  (\alpha\int_\Permemb \beltramigrad  u_t \di \sigma
 + \alpha\int_\Permemb \beltramigrad_y  \chi_0 \nabla u_t \di \sigma \Big)=
\\
-\alpha \Div  \left ( \Big( \int_\Permemb (I -\nu \otimes   \nu)\Big) \nabla u_t+
 \Big ( \int_\Permemb  \beltramigrad_y\chi_0 \di \sigma \Big ) \nabla u_t \right ) =
 \\
-\alpha \Div \left ( \Big (\int_\Permemb (I -(\nu \otimes \nu ) + \beltramigrad_y\chi_0 ) \di \sigma \Big )
\nabla u_t  \right ).
\end{multline}
We define
\begin{equation} \label{matrix-A}
A^0 =\int_Y \!\!\!\lfboth \nabla_y \chi_0 \di y+
\alpha\!\! \int_\Permemb \!\!\beltramigrad_y \chi_1(y,0)\di \sigma
=\int_\Permemb\left\{\alpha\beltramigrad_y \chi_1(y,0)-[\lfboth]\chi_0\otimes\nu\right\}\di\sigma
\,,
\end{equation}
\begin{equation} \label{matrix-B}
 B^0 (t)=\alpha\! \!\int_\Permemb \! \! \beltramigrad_y \chi_{1t} (y, t)  \di \sigma \! +\!\!
 \int_Y \! \! \lfboth \nabla_y \chi_1 (y, t)  \di y
=\!\!\int_\Permemb\!\!\!\left\{\alpha\beltramigrad_y \chi_{1t} (y, t)-[\lfboth] \chi_1 (y, t)\otimes\nu \right\}
\!\!\di \sigma
\end{equation}
and
\begin{equation} \label{matrix-C}
C^0 = \alpha \, \displaystyle \int_\Permemb \left (I -\nu \otimes \nu  +
\beltramigrad_y\chi_0 \right ) \di \sigma=\alpha\int_{\Permemb}\beltramigrad_y(\chi_0+y)\di\sigma.
\end{equation}
Obviously, the matrix $B^0$ is of class $\CC^\infty(0,T)$.
Also, we set
\begin{equation} \label{F}
F \!=\!\Div \!\!\left (\int_Y\lfboth\nabla_y W \di y + \alpha\int_{\Permemb} \beltramigrad_y W_t \di \sigma\right  )\!\!=
\Div\!\!\int_{\Permemb} \!\!\!\left (\alpha\beltramigrad_y W_t -[\lfboth] W\nu\right)\di\sigma.
\end{equation}
The matrices $A^0,B^0$ and $C^0$ are well-defined.
By using the definition of the matrices $A^0,B^0$ and $C^0$,  we obtain that
the two-scale system \eqref{homog-1_1}--\eqref{homog-1_5} can be decoupled and we
get immediately the homogenized equation appearing in problem  \eqref{homog-pb}.
Concerning the boundary and the initial condition, we note that
the first one is simply a direct consequence of the fact that the pair $(u,w)\in {\mathcal V}$,
while the second one depends on the geometry.
Indeed, in the connected/disconnected case, taken into account that $C^0=0$, as follows from
Lemma \ref{r1} below, the homogenized problem does not require any initial condition,
according to the fact that the last equation in \eqref{homog-pb} disappears.
As observed at the beginning of the proof, this corresponds to the fact that
\eqref{homog-1_6} becomes, in this case, an identity.
On the contrary, in the connected-connected case $C^0$ is positive definite, as a
consequence of Lemma \ref{l:l2} below, so that an initial condition in \eqref{homog-pb} is in fact required.

\hfill$\Box$

\begin{Remark}\label{r:r8}
%
%
Assume that $\Omout$ and $\Omint$ are both connected and that the solution $u$
of \eqref{unfold} is sufficiently regular
so that, up to time $t=0$, we have that $u=0$ on $\partial\Om$.
Then, any solution $\big(u(\cdot, 0),w(\cdot,\cdot,0)\big)$ of the system \eqref{homog-1_6}--\eqref{homog-1_5},
complemented with the condition
$u(x,0)=0$ on $\partial\Om$, satisfies
\begin{equation}\label{eq:a30}
u(x,0)=\overline u_0(x)\,,\qquad \beltramigrad_yw(x,y,0)=0\,.
\end{equation}
Indeed, setting $U(x)=u_2(x,0)-u_1(x,0)$ and
$\WW(x,y)=w_2(x,y,0)-w_1(x,y,0)$, where $(u_i,w_i)$, $i=1,2$, are two solutions of the
system \eqref{homog-1_6}--\eqref{homog-1_5} satisfying ${u_1}_{\mid\partial\Om}=0=
{u_2}_{\mid\partial\Om}$, using $U$ and $\WW$ as
testing functions in \eqref{homog-1_6} and \eqref{homog-1_5} (written for $(u_i,w_i)$), respectively,
integrating by parts and subtracting the two equations, it follows
\begin{equation}\label{eq:a31}
\int_{\Omega\times \Gamma}
     |\beltramigrad_y \WW|^2\di x \di \sigma =-\int_{\Omega\times \Gamma}
     \beltramigrad_y \WW\cdot\beltramigrad U\di x \di \sigma
         \end{equation}
and
\begin{multline}\label{eq:a32}
\int_{\Omega\times \Gamma}
     |\beltramigrad U|^2\di x \di \sigma =
   \int_\Om\left(\int_{\Permemb} \beltramigrad U\di \sigma\right)\cdot\nabla U \di x  =
     \\
     -\int_{\Omega}
    \left( \int_{\Permemb}\beltramigrad_y \WW \di \sigma\right)\cdot\nabla U\di x
=-\int_{\Omega\times \Gamma}
     \beltramigrad_y \WW\cdot\beltramigrad U\di x \di \sigma\,.
              \end{multline}
Summing \eqref{eq:a31} and \eqref{eq:a32}, we obtain
\begin{equation*}
\int_{\Omega\times \Gamma}\left|\beltramigrad U+\beltramigrad_y \WW\right|^2\di x \di \sigma=0\,.
\end{equation*}
This implies $\beltramigrad_y \WW=-(I-\nu\otimes\nu)\nabla U=- \beltramigrad_y(y\cdot\nabla U)$, with $U$ depending only on $x$ and where we used the fact that $\beltramigrad_y y=I-\nu\otimes\nu$. Hence, there
exists a $y$-constant function $C(x)$ such that $\WW(x,y)+y\cdot\nabla U(x)=C(x)$ on $\Permemb$; but
exploiting the $Y$-periodicity of $ \WW$ and taking into account the geometrical setting,
we get $\nabla U=0$ in $\Om$ and $W(x,y)=C(x)$ on $\Permemb$.
Therefore, since $U$ vanishes on $\partial\Om$, it follows that $U=0$ in $\Om$. Moreover,
$\beltramigrad_y \WW=0$ on $\Permemb$. Hence, taking into account that ${\overline u_0}_{\mid{\partial\Om}}=0$
and that a pair satisfying \eqref{eq:a30} (that is a pair $(\overline u_0, w)$,
with $\beltramigrad_y w=0$) is a solution of \eqref{homog-1_6}--\eqref{homog-1_5}, the assertion
follows.
\end{Remark}

\begin{lemma}\label{l:l2}
The matrix $C^0$ is symmetric. Moreover, if we are
in the connected/connected case, then the matrix $C^0$ is also positive definite.
\end{lemma}

\begin{proof}
Taking into account \eqref{cell-chi0_2}, let us compute
\begin{multline}\label{eq:a51}
\alpha \int_\Gamma \beltramigrad_y (\chi_0^h+y_h) \cdot \beltramigrad_y (\chi_0^j+y_j) \di \sigma
\\
=\alpha \int_\Gamma \beltramigrad_y (\chi_0^h+y_h) \cdot \beltramigrad_y \chi_0^j \di \sigma+
\alpha \int_\Gamma \beltramigrad_y (\chi_0^h+y_h) \cdot \beltramigrad_y y_j \di \sigma
\\
=\alpha \int_\Gamma \beltramigrad_y \chi_0^h \cdot ({\mathbf{e}}_j-\nu_j\nu)\di \sigma
+\alpha \int_\Gamma ({\mathbf{e}}_h-\nu_h\nu)\cdot ({\mathbf{e}}_j-\nu_j\nu)\di \sigma
\\
=\alpha \int_\Gamma \left(\beltramigrad_y \chi_0^h\right)^{}_j\di \sigma
+\alpha \int_\Gamma (\delta_{hj}-\nu_h\nu_j)\di \sigma=C^0_{hj}\,,
\end{multline}
where the last equality follows from definition \eqref{matrix-C}.
Hence, the symmetry of the matrix $C^0$ is proved.
\noindent
In order to prove the positive definiteness, we calculate
\begin{multline}\label{eq:a53}
\sum_{h,j=1}^{N}C^0_{hj} \xi_h\xi_j =
\alpha\int_\Permemb \sum_{h,j=1}^{N}
\nabla^B_y(\chi_0^h\xi_h+y_h\xi_h)\cdot\nabla^B_y(\chi_0^j\xi_j+y_j\xi_j)\di \sigma
\\
=\alpha\int_\Permemb \big|\sum_{h=1}^{N}\nabla_y^B(\chi_0^h\xi_h+y_h\xi_h)\big|^2\di \sigma
=\alpha\int_\Permemb \left|\nabla_y^B\left(\sum_{h=1}^{N}(\chi_0^h\xi_h+y_h\xi_h)\right)\right|^2\di \sigma
\geq 0\,.
\end{multline}
Assume, by contradiction, that the last integral is equal to zero for a nonzero vector $\xi=(\xi_1,\dots,\xi_N)$;
this implies that, a.e. on $\Permemb$,
$$
\nabla_y^B\left(\sum_{h=1}^{N}(\chi_0^h\xi_h+y_h\xi_h)\right)=0\,,
\qquad\text{which implies that}\qquad
\sum_{h=1}^{N}(\chi_0^h(y)\xi_h+y_h\xi_h)=C\,,
$$
for a suitable constant $C$. However, in this case we have
\begin{equation}\label{eq:a52}
C-\sum_{h=1}^{N}\chi_0^h(y)\xi_h=\sum_{h=1}^{N}y_h\xi_h\,,
\end{equation}
which leads to a contradiction, since the left-hand side of \eqref{eq:a52} is a periodic function on $\Permemb$ while the right-hand side
is not because of our geometrical assumptions. Hence, the last inequality in \eqref{eq:a53} is actually strict and, by standard arguments,
this is enough to prove that the homogenized matrix is positive definite.
\end{proof}

\begin{lemma}\label{r1}
If we are in the connected/disconnected case, then the matrix $C^0=0$.
\end{lemma}

\begin{proof}
Equation \eqref{cell-chi0_1} implies that, up to an additive constant, $\chi^j_0 =-y_j$ on each connected
component of $\Perint\cup\Permemb$ and hence
the matrix $C^0=0$.
\end{proof}

\begin{remark}\label{r:r3}
In the connected/disconnected case, the previous Lemma \ref{r1}
 implies that the principal part of equation \eqref{homog-2}
becomes
\begin{multline*}
-\Div \left ( \Big (\lfav I + \int_Y \lambda \nabla_y \chi_0 \di y  + \alpha \int_\Gamma \beltramigrad_y \chi_1(y,0) \di \sigma \Big ) \nabla u \right )
\\
= -\Div\left ( \Big (\lfav I+  A^0\Big)\nabla u \right ),
\end{multline*}
while this is not the case when $\Omint$ and $\Omout$ are both connected.
Moreover, since $\nabla_y\chi^j_0 =-\ej$ on $\Perint$, we may rewrite the homogenized matrix in the form
\begin{multline}\label{eq:a54}
\lfav I+\int_Y \lfboth \nabla_y\chi_0\di y+ \alpha \int_\Gamma \beltramigrad_y \chi_1(y,0) \di \sigma
\\
=\lfint|\Perint|I+\lfout|\Perout|I+\lfout\int_{\Perout}\nabla_y\chi_0\di y
-\lfint|\Perint|I
+ \alpha \int_\Gamma \beltramigrad_y \chi_1(y,0) \di \sigma
\\
=\int_{\Perout}\lfout(I+\nabla_y\chi_0)\di y
+ \alpha \int_\Gamma \beltramigrad_y \chi_1(y,0) \di \sigma
\\
= \lfout |\Perout|I+\lfout \int_{\Permemb} y\otimes\nu\di\sigma
+ \alpha \int_\Gamma \beltramigrad_y \chi_1(y,0) \di \sigma\,.
\end{multline}
\end{remark}

\begin{remark}\label{r:r4}
We note that in the case of a layered geometry, where the layers are for instance transversal to the direction $\eh$,
a similar argument as the one in Lemma \ref{r1} leads to prove that the matrix $C^0$ is not
identically equal to zero, but it degenerates in the direction $\eh$, since in this case we have
$\chi_0^h(y)=-y_h$ in $\Perint$ (up to an additive constant).
\end{remark}

\begin{lemma}\label{l:l1}
The matrix $\lfav I+A^0$ is symmetric and positive definite.
\end{lemma}

\begin{proof}
Taking into account \eqref{cell-chi0_1}, let us compute
\begin{equation}\label{eq:calc1}
0=-\int_{Y} \Div_y\big(\lfboth\nabla_y(\chi^j_0+y_j)\big)\chi_0^h\di y =
\int_Y\lfboth \nabla_y(\chi^j_0+y_j)\cdot\nabla_y\chi_0^h\di y
+\int_{\Permemb} [\lfboth\nabla_y(\chi^j_0+y_j)\cdot\nu]\chi_0^h\di \sigma\,.
\end{equation}
Moreover, by \eqref{cell-chi0_2}, it follows that
\begin{equation*}
0=-\alpha\int_{\Permemb} \beltrami_y (\chi^h_0+y_h)\chi_1^j(y,0)\di\sigma=
\alpha\int_{\Permemb} \beltramigrad_y(\chi^h_0+y_h)\cdot\beltramigrad_y\chi_1^j(y,0)\di\sigma\,,
\end{equation*}
i.e.
\begin{multline}\label{eq:calc2}
-\alpha\int_{\Permemb}\beltramigrad_y\chi_1^j(y,0)\cdot \beltramigrad_y\chi^h_0\di\sigma
=\alpha\int_{\Permemb}\beltramigrad_y\chi_1^j(y,0)\cdot \beltramigrad_yy_h\di\sigma
\\
=\alpha\int_{\Permemb}\beltramigrad_y\chi_1^j(y,0)\cdot ({\mathbf e}_h-\nu_h\nu)\di\sigma
=\alpha\int_{\Permemb}\Big(\beltramigrad_y\chi_1^j(y,0)\Big)_h\di\sigma\,.
\end{multline}
Next,  by \eqref{cell-chi1_init}, we get
\begin{multline}\label{eq:calc3}
0=-\alpha\int_{\Permemb} \beltrami_y \chi^j_1(y,0)\chi_0^h\di\sigma
-\int_{\Permemb}[\lfboth\nabla_y(\chi^j_0+y_j)\cdot\nu]\chi_0^h\di\sigma
\\
=\alpha\int_{\Permemb} \beltramigrad_y \chi^j_1(y,0)\cdot\beltramigrad_y\chi_0^h\di\sigma
+\int_{Y}\lfboth\nabla_y(\chi^j_0+y_j)\cdot\nabla_y\chi_0^h\di y
\end{multline}
where, in the last equality, we used \eqref{eq:calc1}.
Finally, we check directly
\begin{equation}\label{eq:calc5}
\int_{Y}\lfboth\nabla_y(\chi^j_0+y_j)\cdot\nabla_y y_h\di y
= \int_Y \lfboth\left(\pder{\chi^j_0}{y_h}+\delta_{hj}\right)\di y\,.
\end{equation}
Then, by \eqref{eq:calc2}--\eqref{eq:calc5}, we obtain
\begin{multline}\label{eq:calc6}
\int_{Y}\!\!\lfboth\nabla_y(\chi^j_0+y_j)\cdot\nabla_y (\chi^h_0+y_h)\di y
=\int_{Y}\!\!\lfboth\nabla_y(\chi^j_0+y_j)\cdot\nabla_y \chi^h_0\di y+
\int_{Y}\!\!\lfboth\nabla_y(\chi^j_0+y_j)\cdot\nabla_y y_h\di y
\\
=-\alpha\int_{\Permemb}\beltramigrad_y \chi^j_1(y,0)\cdot\beltramigrad_y \chi^h_0\di\sigma
+\lambda_0\delta_{hj}+\int_Y \lfboth\pder{\chi^j_0}{y_h}\di y
\\
=\alpha\int_{\Permemb}\Big(\beltramigrad_y \chi^j_1(y,0)\Big)_h\di\sigma
+\lambda_0\delta_{hj}+\int_Y \lfboth\pder{\chi^j_0}{y_h}\di y
=\lfav \delta_{hj}+A^0_{jh}\,,
\end{multline}
where, in the last equality, we recall \eqref{matrix-A}. Therefore,
the symmetry of the matrix $\lfav I+A^0$ is proved.

\noindent
In order to prove its positive definiteness, we proceed as follows:
setting $\lfboth_{min} =\min(\lfint,\lfout)$ and
using Jensen's inequality, we obtain
\begin{equation}\label{eq:a26}
\begin{aligned}
&\hphantom{=}
\sum_{h,j=1}^{N}(\lfav I  + A^0)_{hj} \xi_h\xi_j =
\int_Y \lfboth\sum_{h,j=1}^{N}\nabla_y(\chi_0^h\xi_h+y_h\xi_h)\cdot
\nabla_y(\chi_0^j\xi_j+y_j\xi_j)\di y
\\
& \geq \lfboth_{min} \int_Y \big|\sum_{h=1}^{N}\nabla_y(\chi_0^h\xi_h+y_h\xi_h)\Big|^2\di y
\geq \lfboth_{min}\left|\int_Y \sum_{h=1}^{N}\nabla_y(\chi_0^h\xi_h+y_h\xi_h)\di y\right|^2
\\
& = \lfboth_{min}\sum_{j=1}^{N}\left( \sum_{h=1}^{N}(\xi_h\int_Y
\frac{\partial\chi_0^h}{\partial y_j}\di y+\delta_{hj}\xi_h)\right)^2
\\
&= \lfboth_{min}\sum_{j=1}^{N}\left( \sum_{h=1}^{N}\xi_h\int_{\partial Y}
\chi_0^h\,n_j\di \sigma+\xi_j\right)^2
= \lfboth_{min}|\xi|^2
\end{aligned}
\end{equation}
where we have denoted by $n=(n_1,\dots,n_N)$ the outward unit normal to $\partial Y$.
Indeed, we
remark that the last integral vanishes because of the periodicity of the cell
function $\chi_0^h$.
This proves that the homogenized matrix $\lfav I+A^0$
is positive definite and concludes the lemma.
\end{proof}

\begin{remark}\label{r:r2}
Notice that in the connected-connected or connected-disconnected geometry, since  $C^0$ and/or $A^0$
are positive definite matrices, problem \eqref{homog-pb} is well-posed (see \cite{Amar:Andreucci:Bisegna:Gianni:2004b}
for the case $C^0=0$, while the case $C^0\not=0$ will be treated in a forthcoming paper).
\end{remark}

\section{Other scalings}
\label{s:scale}

In this section, we will consider the homogenization of our microscopic problem
\eqref{eq:PDEink}--\eqref{eq:InitDatak}, whose weak formulation is the following
\begin{multline}\label{eq:weak_solk}
  \int_{0}^{T}\!\!\int_{\Om} \lfbothe \nabla u_{\eps}\cdot\nabla\Phi \di x\di t
  -{\eps^k\alpha}\int_0^T\!\!\int_{\Memb} \beltramigrad u_{\eps}\cdot \beltramigrad \Phi_t \di\sigma\di t
  \\
  ={\eps^k\alpha}\int_{\Memb} \beltramigrad\overline u_{0\eps}\cdot\beltramigrad\Phi(x,0)\di\sigma
  +\int_{0}^{T}\!\!\int_{\Om} f \Phi\di x\di t\,,
\end{multline}
for every test function $\Phi\in \CC^\infty(\overline\Om_T)$ such that $\Phi$
has compact support in $\Om$ for every
$t\in[0,T)$ and $\Phi(\cdot,T)=0$ in $\Om$. The corresponding energy estimate is
\begin{equation}
\label{eq:energyk}
\int_{0}^{T}\int_{\Om} \abs{\nabla u_\eps}^{2}
  \di x\di\tau
  +  \eps^k\sup_{t\in(0,T)}\int_{\Memb} \abs{\beltramigrad u_\eps}^{2} \di \sigma \le \gamma\,,
\end{equation}
where we used the assumption $\overline u_{0\eps}= \eps^{(1-k)/2}_{} \overline u_0$, with
$\overline u_0\in H^1_0(\Om)\cap H^2(\Om)$.

In Section \ref{s:homog} we have studied the case $k=1$, which seems to be the most physical one,
since it is the only case where the limit problem keeps memory
of the physical properties of the active membranes.
Here, we will consider the two cases $k>1$ and $k<1$, respectively,
where no memory of $\alpha$ remains in
the homogenized equation. However, in some cases, memory of the geometry or of $\lfint$ is kept.
In particular, when $k<1$ and we are in the connected/connected case, the limit solution is identically
equal to $0$, even if the initial datum and the source are not null; notice that this is not the case in the
connected/disconnected geometry.

\subsection{Case $k<1$: proof of Theorems \ref{t:casoKpiccolo} and \ref{t:casoKpiccolo_bis}}
\label{ss:minore1}
When $k<1$, from the energy estimate \eqref{eq:energyk}, we obtain that the
estimate \eqref{eq:energy} is satisfied as well, so that all the results in Lemma \ref{conv}
hold. In particular, we still obtain that $\unfop(\nabla u_\eps)\wto \nabla u+\nabla_y w$
weakly in $L^2(\Om_T\times Y)$ and
$\btsunfop(\beltramigrad u_\eps)\wto \beltramigrad u+
\beltramigrad_y w$ weakly in $L^2(\Om_T\times \Permemb)$.

Moreover, let us take in the weak formulation \eqref{eq:weak_solk} a test function of the type
$\Phi(x,t)= \eps\varphi(x,t)\psi(\eps^{-1}x)$,
with  $\varphi\in \CC^\infty(\overline\Om_T)$ such that $\varphi$
has compact support in $\Om$ for every $t\in[0,T)$, $\varphi(\cdot,T)=0$ in $\Om$ and
$\psi\in \CC^\infty_\#(Y)$ with $supp (\psi)\subset\subset \Perint\cup\Perout$.
Unfolding and then passing to the limit for $\eps\to 0$, we obtain
$$
\int_{0}^{T}\!\!\int_{\Om\times Y}\lfboth(\nabla u+\nabla_y w)\cdot\nabla_y\psi\varphi
 \di y\di x\di t=0\,,
$$
i.e.,
\begin{equation}\label{eq:a25}
-\Div_y\big(\lfboth(\nabla u+\nabla_y w)\big)=0\,,\qquad \text{in $\Perint\cup\Perout$,}
\end{equation}
with $w\in L^2(\Om_T;H^1_\#(Y))$.
Next, take in the weak formulation \eqref{eq:weak_solk} a test function of the type
$\Phi(x,t)= \eps^{2-k}\varphi(x,t)\psi(\eps^{-1}x)$,
with  $\varphi$ as before and $\psi\in \CC_\#^\infty(Y)$.
Unfolding and passing to the limit for $\eps\to 0$, we obtain, owing
to our assumption $\overline u_{0\eps}=\eps^{(1-k)/2}_{} \overline u_0$, $$
\alpha\int_0^T\!\!\int_{\Om\times \Permemb} (\beltramigrad u+\beltramigrad_y w)
\cdot \beltramigrad_y \psi \varphi_t \di\sigma\di x\di t =0\,,
$$
i.e.,
$$
\begin{aligned}
& -\alpha\beltramidiv_y(\beltramigrad u_t+\beltramigrad_y w_{t})=0\,,\qquad &\text{on $\Permemb\times(0,T)$,}\\
&-\alpha\beltramidiv_y(\beltramigrad u(0)+\beltramigrad_y w(0))=0\,,\qquad &\text{on $\Permemb$,}
\end{aligned}
$$
which gives
$$
-\alpha\beltramidiv_y(\beltramigrad u+\beltramigrad_y w)=0\,,\qquad \text{on $\Permemb\times(0,T)$.}
$$
This implies that we can factorize $w(x,t,y)=\chi_0^j(y)\partial_j u(x,t)$, where
$\chi_0^j$ satisfies \eqref{cell-chi0_1}--\eqref{cell-chi0_2} and \eqref{cell-chi0_3}.
Moreover,condition \eqref{cell-chi0_5} is automatically satisfied in the connected/connected case,
since $\Permemb$ has only one connected component, while it is satisfied in the connected/disconnected case
thanks to \cite[Proposition 2.6]{Amar:Andreucci:Gianni:Timofte:2017B}.
\medskip

Now we have to proceed separately in the two cases,
since we have to consider different test functions. Indeed, the test function which we will use in
the connected/connnected case will not give any information in the connected/disconnected case
(see Remark \ref{r:r11}), while the test function
which we will use in this last case cannot be constructed in the connected/connected case.

\medskip

{\it{Proof of Theorem \ref{t:casoKpiccolo}.}}
Let us take in the weak formulation \eqref{eq:weak_solk} a test function of the type
$\Phi(x,t)= \eps^{1-k}\varphi(x,t)$, with  $\varphi\in \CC^\infty(\overline\Om_T)$ such that $\varphi$
has compact support in $\Om$ for every $t\in[0,T)$, $\varphi(\cdot,T)=0$ in $\Om$.
Unfolding and passing to the limit for $\eps\to 0$, we obtain
$$
-\alpha\int_0^T\!\!\int_{\Om\times\Permemb} (\beltramigrad u+\beltramigrad_y w)
\cdot \beltramigrad \varphi_t \di x\di\sigma\di t=0\,,
$$
i.e.,
$$
\begin{aligned}
& -\alpha\Div\int_{\Permemb}(\beltramigrad u_t+\beltramigrad_y w_{t})\di\sigma=0\,,\qquad &\text{on $\Om_T$,}\\
&-\alpha\Div\int_{\Permemb}(\beltramigrad u(0)+\beltramigrad_y w(0))=0\,,\qquad &\text{on $\Om$,}
\end{aligned}
$$
which gives
$$
-\alpha\Div\left(\int_{\Permemb}(\beltramigrad u+\beltramigrad_y w)\di\sigma\right)=
0\,,\qquad\text{on $\Om_T$.}
$$
Inserting in the previous equation the factorization of $w$ in terms of the cell functions, we
obtain
\begin{equation}\label{eq:a21}
-\Div\left(\big(\alpha\int_{\Permemb}(I-\nu\otimes    \nu+\beltramigrad_y\chi_0)\di\sigma\big)
\nabla u\right)=0\,,
\end{equation}
which, recalling \eqref{matrix-C}, can be rewritten as $-\Div(C^0\nabla u)=0$.
Taking into account that $u_{\mid\partial\Om}=0$ and that in the connected/connected case
the matrix $C^0$ is positive definite by Lemma \ref{l:l2}, it follows that $u\equiv 0$ in $\Om$,
so that the whole sequence $\{u_\eps\}$ converges to zero.

\hfill $\Box$

\begin{remark}\label{r:r10}
Notice that the previous result holds true though a non-zero source term appears in
\eqref{eq:PDEink}. Moreover, according to \eqref{eq:ts_norm_bound} and
from the energy estimate \eqref{eq:energyk}, we also get
\begin{equation*}
\int_{0}^{T}\!\!\int_{\Om\times\Permemb} \abs{\btsunfop (\beltramigrad u_\eps)}^2\di\sigma\di t
\leq\eps\int_{0}^{T}\!\!\int_{\Gamma^\eps} \abs{\nabla^B u_\eps}^{2} \di \sigma \di t\le \gamma \eps^{1-k}
\to 0\,,\qquad\text{for $\eps\to 0$,}
\end{equation*}
which gives a rate of convergence to zero of $\Vert \btsunfop (\beltramigrad u_\eps)\Vert_{L^2(\Om_T\times\Permemb)}$.
Hence, passing to the limit and taking into account the lower semicontinuity of the
norm with respect to the weak convergence, it follows that
\begin{equation}\label{eq:a68}
\Vert \beltramigrad u+\beltramigrad_y w\Vert^{}_{L^2(\Om_T\times\Permemb)} \leq
\liminf_{\eps\to 0}\Vert \btsunfop(\beltramigrad u_\eps)\Vert^{}_{L^2(\Om_T\times\Permemb)}
=0\,.
\end{equation}
\end{remark}

\begin{remark}\label{r:r11}
Notice that when we are in the connected/disconnected geometrical setting, as pointed out
in Lemma \ref{r1}, $\chi_0^j=-y_j$, up to an additive constant on each connected component of $\Perint\cup\Permemb$,
so that the matrix $C^0=0$. Hence, equations
\eqref{eq:a21} and \eqref{eq:a68} do not give any information on $u$.
\end{remark}

{\it Proof of Theorem \ref{t:casoKpiccolo_bis}.}
As in \cite[Proof of Lemma 4.1]{Cioranescu:Damlamian:Li:2013}, let us take in the weak formulation \eqref{eq:weak_solk} a test function of the type
$\Phi(x,t)=\varphi(x,t)\big(1-\psi(\eps^{-1}x)\big)+\average(\varphi)\psi(\eps^{-1}x)$,
with  $\varphi\in \CC^\infty(\overline\Om_T)$ such that $\varphi$
has compact support in $\Om$ for every $t\in[0,T)$, $\varphi(\cdot,T)=0$ in $\Om$,
and $\psi\in \CC^\infty_c(Y)$, with $\psi=1$ on $\Perint\cup\Permemb$.
Thus we obtain
\begin{multline*}
  \int_{0}^{T}\!\!\int_{\Om} \lfbothe \nabla u_{\eps}\cdot\big\{\nabla\varphi(1-\psi)+\frac{1}{\eps}(\average(\varphi)-\varphi)\nabla_y\psi\big\}
  \di x\di t
  \\
  =\int_{0}^{T}\!\!\int_{\Om} f \big\{\varphi(1-\psi)+\average(\varphi)\psi\big\}\di x\di t\,.
\end{multline*}
Then, unfolding and passing to the limit for $\eps\to 0$, it follows
\begin{equation}\label{eq:a24}
  \int_{0}^{T}\!\!\int_{\Om} \!\!\int_Y\lfboth (\nabla u+\nabla_y\genfun)
  \cdot\big\{\nabla\varphi(1-\psi)-(y_M\cdot\nabla\varphi)\nabla_y\psi\big\}
  \di x\di y\di t
  =\int_{0}^{T}\!\!\int_{\Om} f \varphi\di x\di t\,,
\end{equation}
where we have taken into account that $\varphi(1-\psi)+\average(\varphi)\psi\to \varphi$, strongly in $L^2(\Om_T)$
and \eqref{eq:a23} holds.
Notice that, by the identity $\nabla\varphi\psi+(y_M\cdot\nabla\varphi)\nabla_y\psi=\nabla_y\big((y_M\cdot\nabla\varphi)\psi\big)$,
equality \eqref{eq:a24} can be rewritten in the form
\begin{multline*}
  \int_{0}^{T}\!\!\int_{\Om} \!\!\int_Y\lfboth (\nabla u+\nabla_y\genfun)
  \cdot\nabla\varphi\di x\di y\di t -  \int_{0}^{T}\!\!\int_{\Om} \!\!\int_Y\lfboth (\nabla u+\nabla_y\genfun) \nabla_y\big((y_M\cdot\nabla\varphi)\psi\big)  \di x\di y\di t
  \\
  =\int_{0}^{T}\!\!\int_{\Om} f \varphi\di x\di t\,,
\end{multline*}
which becomes
\begin{multline*}
  \int_{0}^{T}\!\!\int_{\Om} \!\!\int_Y\lfboth (\nabla u+\nabla_y\genfun)
  \cdot\nabla\varphi\di x\di y\di t
  + \int_{0}^{T}\!\!\int_{\Om} \!\!\int_{\Permemb}[\lfboth (\nabla u+\nabla_y\genfun)\cdot\nu]
 (y_M\cdot\nabla\varphi)\di x\di \sigma\di t
  \\
  =\int_{0}^{T}\!\!\int_{\Om} f \varphi\di x\di t\,,
\end{multline*}
as a consequence of \eqref{eq:a25}. Finally, taking into account the factorization of $\genfun(x,t,y) = \chi_0^j(y)\partial_ju(x,t)$,
with $\chi_0$ satisfying \eqref{cell-chi0_1}--\eqref{cell-chi0_3} and recalling that in the connected/dis\-con\-nected case, $\nabla_y\chi_0(y)=-I$
on $\Perint$, we obtain the homogenized equation
$$
-\Div\left(\Big(\int_{\Perout}\!\!\!\lfout(I+\nabla_y\chi_0)\di y+
\int_{\Permemb}\!\!\!\lfout\Big((\nu+(\nabla_y\chi_0)^{\rm{out}}\nu)\otimes y_M\Big)\di\sigma\Big)
\nabla u\right)\!=\!f\,.
$$
In order to prove that the homogenized matrix
$$
A^{\rm{hom}}:=\int_{\Perout}\lfout(I+\nabla_y\chi_0)\di y+\int_{\Permemb}\lfout\big((\nu+(\nabla_y\chi_0)^{\rm{out}}\nu)
\otimes y_M\big)\di\sigma
$$
is symmetric and positive definite, we proceed as follows. By \eqref{eq:calc1}, we get
\begin{multline*}
\int_{Y}\!\!\lfboth\nabla_y(\chi^j_0+y_j)\cdot\nabla_y (\chi^h_0+y_h)\di y
\\
=\int_{Y}\!\!\lfboth\nabla_y(\chi^j_0+y_j)\cdot\nabla_y \chi^h_0\di y+
\int_{Y}\!\!\lfboth\nabla_y(\chi^j_0+y_j)\cdot\nabla_y y_h\di y
\\
=-\int_{\Permemb}[\lfboth \nabla_y(\chi^j_0+y_j)\cdot \nu] \chi^h_0\di\sigma
+\int_Y \lfboth\big(\pder{\chi^j_0}{y_h}+\delta_{hj}\big)\di y
\\
=\sum_{i=1}^{m}\int_{\Permemb_i}\lfout\big((\nabla_y\chi^j_0)^{\rm{out}}+\ej\big)\cdot \nu (y_h+c_i)\di\sigma
+\int_{\Perout} \lfout\big(\pder{\chi^j_0}{y_h}+\delta_{hj}\big)\di y
\\
=\int_{\Permemb}\lfout\big((\nabla_y\chi^j_0)^{\rm{out}}\cdot \nu+\nu_j\big) (y_M)_h\di\sigma
+\int_{\Perout} \lfout\big(\pder{\chi^j_0}{y_h}+\delta_{hj}\big)\di y
=A^{\rm{hom}}_{jh}\,,
\end{multline*}
where $\Permemb_i$, $i=1,\dots,m$, are the connected component of $\Permemb$ and,
in the fourth equality, $y_h+c_i$ has been replaced with $(y_M)_h$, since
we have taken into account that $\int_{\Permemb}\nu_j\di\sigma=0$ and \eqref{cell-chi0_5} holds.
This proves the symmetry; the positive definiteness now follows directly from \eqref{eq:a26}.
Hence, $u$ is uniquely determined, which implies that the whole sequence $\{u_\eps\}$ converges.

\hfill$\Box$
\medskip

\begin{remark}\label{r:r12}
Notice that the homogenized solution $u$ does not depend on $\alpha$ nor on $\lfint$; i.e.,
it does not depend on the physical properties of the interface
and of the inclusions, being affected only by the physical properties of the surrounding matrix.
\end{remark}

\subsection{Case $k>1$: proof of Theorem \ref{t:casoKgrande}}
\label{ss:maggiore1}
Setting $v_\eps = \eps^{(k-1)/2}u_\eps$, from the energy estimate \eqref{eq:energyk}
it follows that $\btsunfop(\beltramigrad v_\eps)$
is bounded in $L^2(\Om_T\times\Permemb)$
and, up to a subsequence, $\unfop(\nabla u_\eps)\wto \nabla u+\nabla_y w$ weakly
in $L^2(\Om_T\times Y)$.
Then, taking in the weak formulation \eqref{eq:weak_solk} a test function of the type
$\Phi(x,t)=\varphi(x,t)+\eps\phi(x,t)\psi(x,\eps^{-1}x,t)$ with $\varphi,\phi\in \CC^\infty(\Om_T)$ such that $\varphi,\phi$
have compact support in $\Om$ for every $t\in[0,T)$, $\varphi(\cdot,T)=\phi(\cdot,T)=0$ in $\Om$ and
$\psi\in \CC^\infty_\#(Y)$, unfolding and then passing to the limit, using also
the boundedness of $\btsunfop(\beltramigrad v_\eps)$, we obtain
$$
  \int_{0}^{T}\!\!\int_{\Om\times Y} \lfboth (\nabla u+\nabla_y w)\cdot(\nabla\varphi
   +\nabla_y\psi\,\phi)\di x\di y\di t
  =\int_{0}^{T}\!\!\int_{\Om\times Y} f \varphi\di x\di t\,,
$$
which gives
$$
\begin{aligned}
-\Div\left(\int_{Y}\lfboth (\nabla u+\nabla_y w)\di y\right)=f\,,
\\
-\Div_y\big(\lfboth (\nabla u+\nabla_y w)\big)=0\,.
\end{aligned}
$$
Hence, we are led to the standard homogenized two-scale system which can be obtained
in the case of perfect contact
(i.e., when $\alpha=0$ in \eqref{eq:Circuitk}). Here, the time-dependence is
only parametric through the source $f$.

Factorizing $w(x,t,y)=\widetilde\chi^j_0(y)\partial_ju(x,t)$
and inserting in the previous set of equations, it follows that the homogenized function
$u$ is the solution of the problem
$$
\begin{aligned}
 -\Div\left((\int_{Y}\lfboth(I+\nabla_y\widetilde\chi_0)\di y\big)\nabla u\right)& =f\,,\qquad
&\text{in $\Om_T$}\,;\\
u& =0\,,\qquad &\text{in $\partial\Om\times(0,T)$}\,,
\end{aligned}
$$
where $\widetilde\chi^j_0\in H^1_{\#}(Y)$, $j=1,\dots,N$, are $Y$-periodic functions with null mean average satisfying
$$
-\Div_y\big(\lfboth\nabla_y (y_j+\widetilde\chi^j_0)\big)=0\,,\qquad \text{in $Y$,}
$$
which can be rewritten also in the form
\begin{alignat}2
\label{eq:a33} 
-\Div_y\big(\lfboth\nabla_y (y_j+\widetilde\chi^j_0)\big)&=0\,,\qquad &\text{in $\Perint\cup\Perout$;}
\\
\label{eq:a34} 
[\lfboth\nabla_y (y_j+\widetilde\chi^j_0)\cdot\nu]&=0\,,\qquad &\text{on $\Permemb$.}
\end{alignat}
Notice that, by standard results, the homogenized matrix is symmetric and positive definite,
so that $u$ is uniquely determined, which implies that the whole sequence $\{u_\eps\}$ converges;
moreover, the geometrical setting does not play any role and the result holds both
in the connected/connected and in the connected/disconnected case.

\hfill$\Box$


\bigskip

\begin{thebibliography}{10}

\bibitem{Allaire:1992}
G.~Allaire.
\newblock Homogenization and two-scale convergence.
\newblock {\em SIAM J. Math. Anal.}, 23:1482--1518, 1992.

\bibitem{Amar:Andreucci:Bellaveglia:2015B}
M.~Amar, D.~Andreucci, and D.~Bellaveglia.
\newblock Homogenization of an alternating \textsc{R}obin–-\textsc{N}eumann
  boundary condition via time-periodic unfolding.
\newblock {\em Nonlinear Analysis: Theory, Methods and Applications},
  153:56--77, 2017.

\bibitem{Amar:Andreucci:Bellaveglia:2015A}
M.~Amar, D.~Andreucci, and D.~Bellaveglia.
\newblock The time-periodic unfolding operator and applications to parabolic
  homogenization.
\newblock {\em Atti Accad. Naz. Lincei Rend. Lincei Mat. Appl.}, 2017, To
  appear.

\bibitem{Amar:Andreucci:Bisegna:Gianni:2003c}
M.~Amar, D.~Andreucci, P.~Bisegna, and R.~Gianni.
\newblock Evolution and memory effects in the homogenization limit for
  electrical conduction in biological tissues: the $1$-d case.
\newblock In {\em Proceedings 16th AIMETA Congress of Theoretical and Applied
  Mechanics}. 2003.

\bibitem{Amar:Andreucci:Bisegna:Gianni:2003a}
M.~Amar, D.~Andreucci, P.~Bisegna, and R.~Gianni.
\newblock Homogenization limit for electrical conduction in biological tissues
  in the radio-frequency range.
\newblock {\em Comptes Rendus Mecanique}, 331:503--508, 2003.
\newblock Elsevier.

\bibitem{Amar:Andreucci:Bisegna:Gianni:2004b}
M.~Amar, D.~Andreucci, P.~Bisegna, and R.~Gianni.
\newblock An elliptic equation with history.
\newblock {\em C.\ R.\ Acad.\ Sci.\ Paris, Ser.\ I}, 338:595--598, 2004.
\newblock Elsevier.

\bibitem{Amar:Andreucci:Bisegna:Gianni:2004a}
M.~Amar, D.~Andreucci, P.~Bisegna, and R.~Gianni.
\newblock Evolution and memory effects in the homogenization limit for
  electrical conduction in biological tissues.
\newblock {\em Mathematical Models and Methods in Applied Sciences},
  14:1261--1295, 2004.
\newblock World Scientific.

\bibitem{Amar:Andreucci:Bisegna:Gianni:2006a}
M.~Amar, D.~Andreucci, P.~Bisegna, and R.~Gianni.
\newblock On a hierarchy of models for electrical conduction in biological
  tissues.
\newblock {\em Mathematical Methods in the Applied Sciences}, 29:767--787,
  2006.

\bibitem{Amar:Andreucci:Bisegna:Gianni:2009}
M.~Amar, D.~Andreucci, P.~Bisegna, and R.~Gianni.
\newblock Exponential asymptotic stability for an elliptic equation with memory
  arising in electrical conduction in biological tissues.
\newblock {\em Euro. Jnl. of Applied Mathematics}, 20:431--459, 2009.

\bibitem{Amar:Andreucci:Bisegna:Gianni:2009a}
M.~Amar, D.~Andreucci, P.~Bisegna, and R.~Gianni.
\newblock Stability and memory effects in a homogenized model governing the
  electrical conduction in biological tissues.
\newblock {\em J. Mechanics of Material and Structures}, (2) 4:211--223, 2009.

\bibitem{Amar:Andreucci:Bisegna:Gianni:2010}
M.~Amar, D.~Andreucci, P.~Bisegna, and R.~Gianni.
\newblock Homogenization limit and asymptotic decay for electrical conduction
  in biological tissues in the high radiofrequency range.
\newblock {\em Communications on Pure and Applied Analysis}, (5) 9:1131--1160,
  2010.

\bibitem{Amar:Andreucci:Bisegna:Gianni:2013}
M.~Amar, D.~Andreucci, P.~Bisegna, and R.~Gianni.
\newblock A hierarchy of models for the electrical conduction in biological
  tissues via two-scale convergence: The nonlinear case.
\newblock {\em Differential and Integral Equations}, (9-10) 26:885--912, 2013.

\bibitem{Amar:Andreucci:Gianni:2014b}
M.~Amar, D.~Andreucci, and R.~Gianni.
\newblock Asymptotic decay under nonlinear and noncoercive dissipative effects
  for electrical conduction in biological tissues.
\newblock {\em Nonlinear Differ. Equ. Appl.}, (4)23:48, 2016.

\bibitem{Amar:Andreucci:Gianni:2014a}
M.~Amar, D.~Andreucci, and R.~Gianni.
\newblock Exponential decay for a nonlinear model for electrical conduction in
  biological tissues.
\newblock {\em Nonlinear Analysis, Theory, Methods and Applications},
  131:206--228, 2016.

\bibitem{Amar:Andreucci:Gianni:Timofte:2017B}
M.~Amar, D.~Andreucci, R.~Gianni, and C.~Timofte.
\newblock Well-posedness of two pseudo-parabolic problems for electrical
  conduction in heterogenous media.
\newblock To appear.

\bibitem{Amar:Gianni:2016B}
M.~Amar and R.~Gianni.
\newblock Error estimate for a homogenization problem involving the
  \textsc{L}aplace-\textsc{B}eltrami operator.
\newblock {\em Mathematics and Mechanics of Complex Systems}, (1) 6:41--59,
  2018.

\bibitem{Amar:Gianni:2016A}
M.~Amar and R.~Gianni.
\newblock Laplace-\textsc{B}eltrami operator for the heat conduction in polymer
  coating of electronic devices.
\newblock {\em Discrete and Continuous Dynamical System - Series B},
  (4)23:1739--1756, 2018.

\bibitem{Amar:Gianni:2016C}
M.~Amar and R.~Gianni.
\newblock Existence, uniqueness and concentration for a system of \textsc{PDE}s involving the
      \textsc{L}aplace-\textsc{B}eltrami operator.
\newblock {\em  Submitted} (2018).

\bibitem{Cioranescu:Damlamian:Griso:2002}
D.~Cioranescu, A.~Damlamian, and G.~Griso.
\newblock Periodic unfolding and homogenization.
\newblock {\em Comptes Rendus Mathematique}, 335(1):99--104, 2002.

\bibitem{Cioranescu:Damlamian:Griso:2008}
D.~Cioranescu, A.~Damlamian, and G.~Griso.
\newblock The periodic unfolding method in homogenization.
\newblock {\em SIAM Journal on Mathematical Analysis}, 40(4):1585--1620, 2008.

\bibitem{Cioranescu:Damlamian:Li:2013}
D.~Cioranescu, A.~Damlamian, and T.~Li.
\newblock Periodic homogenization for inner boundary conditions with
  equi-valued surfaces: the unfolding approach.
\newblock {\em Chinese Annals of Mathematics, Series B}, 34B(2):213--236, 2013.

\bibitem{Cioranescu:Donato:Zaki:2006B}
D.~Cioranescu, P.~Donato, and R.~Zaki.
\newblock Periodic unfolding and Robin problems in perforated domains.
\newblock {\em Comptes Rendus Math\'{e}matique}, 342 (1):469--474, 2006.

\bibitem{Cioranescu:Donato:Zaki:2006A}
D.~Cioranescu, P.~Donato, and R.~Zaki.
\newblock The periodic unfolding method in perforated domains.
\newblock {\em Portugaliae Mathematica}, 63(4):467--496, 2006.

\bibitem{Dehghani:Soni:2005}
H.~Dehghani and N.~K. Soni.
\newblock Electrical impedance spectroscopy: theory.
\newblock In K.~D. Paulsen, P.~M. Meaney, and L.~C. Gilman, editors, {\em
  Alternative breast imaging: Four model-based approaches}, pages 85--105.
  Springer, 2005.

\bibitem{Donato:Yang:2012}
P.~Donato and Z.~Yang.
\newblock The periodic unfolding method for the wave equation in domains with
  holes.
\newblock {\em Adv. Math.Sci. Appl.}, 22:521--551, 2012.

\bibitem{Sebnem:2010}
S.~Kemaloglu, G.~Ozkoc, and A.~Aytac.
\newblock Thermally conductive boron nitride/sebs/eva ternary
  composites:processing and characterisation.
\newblock {\em Polymer Composites (Published online on www.interscience.
  wiley.com, 2009, Society of Plastic Engineers)}, pages 1398--1408, 2010.

\bibitem{Warangkhana:2012}
W.~Phromma, A.~Pongpilaipruet, and R.~Macaraphan.
\newblock {\em Preparation and Thermal Properties of PLA Filled with Natural
  Rubber-PMA Core-Shell/Magnetite Nanoparticles, {\rm European Conference; 3rd,
  Chemical Engineering, Recent Advances in Engineering}}.
\newblock Paris, 2012.

\bibitem{Khan:2011}
K.~M. Shahil and A.~A. Balandin.
\newblock Graphene-based nanocomposites as highly efficient thermal interface
  materials.
\newblock {\em Graphene Based Thermal Interface Materials}, pages 1--18, 2011.

\end{thebibliography}

\end{document}